\documentclass{amsart}
\usepackage{amsmath, amssymb, amsthm, mathrsfs, multicol}
\usepackage{palatino}
\usepackage{tikz}
\usepackage{tikz-cd}
\usetikzlibrary{arrows, calc}

\usepackage[margin=1.5in]{geometry}

\newtheorem{theorem}{Theorem}[section]
\newtheorem{lemma}[theorem]{Lemma}
\newtheorem{definition}[theorem]{Definition}
\newtheorem{proposition}[theorem]{Proposition}
\newtheorem{corollary}[theorem]{Corollary}

\def\epsilon{\varepsilon}
\def\phi{\varphi}
\def\emptyset{\varnothing}

\def\R{{\mathbb R}}

\def\F{{\mathbb F}}
\def\N{{\mathbb N}}

\def\<{\left\langle}
\def\>{\right\rangle}

\def\depth{\operatorname{depth}}

\def\ol{\overline}
\def\wt{\widetilde}

\def\on{\operatorname}

\def\AA{\mathcal{A}}

\def\st{\operatorname{str}}

\def\downset{\downarrow\!}

\numberwithin{equation}{section}

\title{Stability of stretched root systems, root posets, and shards}
\author{Will Dana}

\begin{document}
\begin{abstract}

Inspired by the infinite families of finite and affine root systems, we consider a ``stretching'' operation on general crystallographic root systems which, on the level of Coxeter diagrams, replaces a vertex with a path of unlabeled edges. We embed a root system into its stretched versions using a similar operation on individual roots. For a fixed root, we study the growth of two associated structures as we lengthen the stretched path: the downset in the root poset (in the sense of Bj\"orner and Brenti \cite{bjornerbrenti}) and the arrangement of shards, introduced by Nathan Reading. We show that both eventually admit a uniform description, and deduce enumerative consequences: the size of the downset is eventually a polynomial, and the number of shards grows exponentially.

\end{abstract}
\maketitle

\section{Introduction}
\label{intro}

In many questions about root systems, Coxeter groups, and related objects, the type $A$ family is foundational, and often more easily resolved than the general case. In particular, type $A$ is in some sense foundational to the classification of finite and affine root systems: just as the type $A$ Coxeter diagrams are simply paths, almost\footnote{With the exception of the dihedral systems $I_2(m)$.} every infinite family is described by Coxeter diagrams obtained by inserting paths into a fixed diagram.

In this paper, we study a generalization of this kind of family to arbitrary Coxeter diagrams, using a \textbf{stretching} operation: 
\begin{definition}
\label{stretchdef}
Let $G$ be a Coxeter diagram, $x$ a vertex of $G$, and $L_x\sqcup R_x$ a partition of the neighbors of $x$ into two subsets. We call $x$ an \textbf{elastic vertex}, $L_x$ and $R_x$ its \textbf{left} and \textbf{right neighbors} respectively, and the tuple $(x, L_x, R_x)$ \textbf{elastic data}.

Then the \textbf{$n$-stretched diagram} $\st_n(G)$ is obtained by:
\begin{itemize}
\item replacing the vertex $x$ with vertices $x_0,\ldots, x_n$;
\item replacing the edges between $x$ and $L_x$ with correspondingly labeled edges between $x_0$ and $L_x$; 
\item replacing the edges between $x$ and $R_x$ with correspondingly labeled edges between $x_n$ and $R_x$;
\item and inserting unlabeled edges between $x_i$ and $x_{i+1}$ for $0\leq i < n$.
\end{itemize}

We call the subdiagram induced by $x_0,\ldots, x_n$ the \textbf{stretched path}.
\end{definition}

If $\Phi$ is a crystallographic root system associated to $G$, for given elastic data we similarly define the \textbf{$n$-stretched root system} $\st_n(\Phi)$ associated to $\st_n(G)$ (Definition \ref{stretchingrootsystems}).

In looking at a family of stretches of a diagram, two natural questions arise:
\begin{itemize}
\item What attributes of the family stabilize as $n$ becomes large?
 
\item Can we isolate aspects of the family's behavior that resemble the good behavior of the $A_n$ family?
\end{itemize}

However, apart from work of Hepworth \cite{hepworth} showing homological stability for Coxeter groups in stretched families with $R_x = \emptyset$, this perspective appears unexplored. We examine these questions in the context of combinatorial properties of roots. 

\subsection{The root poset}

One way of relating the different stretches of a root system is to embed the shorter stretches into the longer ones. To facilitate this, we write each root as a linear combination of simple roots. Since simple roots correspond to vertices of the diagram, we consider roots to be integer-valued functions on the vertices. 

\begin{definition}
\label{rootstretchdef}
Let $\alpha$ be a integer-valued function on the vertices of a diagram $G$ with elastic data $(x, L_x, R_x)$. Then $\st_n(\alpha)$ is the function on $\st_n(G)$ with value $\alpha(x)$ at all the $x_i$, and with the same values as $\alpha$ elsewhere.
\end{definition}

\begin{proposition}[Proposition \ref{stretchingroots}]
Let $\alpha$ be a root of $\Phi$. Then $\st_n(\alpha)$ is a root of $\st_n(\Phi)$.
\end{proposition}

We thus focus on how data associated to a root grows and stabilizes as we stretch the root. We first consider the \textbf{root poset}, which is an order on the positive roots of $\Phi$ analogous to the weak order on a Coxeter group $W$. In particular, just as reduced words of $W$ correspond to saturated chains based at the identity in the weak order, saturated chains based at simple roots in the root poset correspond to \textbf{reduced expressions}: minimal-length expressions for roots in terms of simple reflections applied to a simple root.

\begin{definition}[\cite{bjornerbrenti}, Definition 4.6.3 and Lemma 4.6.4]
Let $\Phi$ be a root system and let $W$ be the associated Coxeter group. For positive roots $\beta, \gamma\in \Phi$, we say that $\beta \leq \gamma$ in the \textbf{root poset}\footnote{Note that there is a more commonly used definition of the root poset \cite[Definition 5.1.1]{armstrong}, which is a refinement of this poset, and which we do not consider here.} if there exist simple reflections $s_{i_1},\ldots, s_{i_k}\in W$ with associated simple roots $\alpha_{i_1},\ldots, \alpha_{i_k}$ such that
\[
\gamma = s_{i_k}\cdots s_{i_1}(\beta)
\]
and for all $1\leq j\leq k$,
\[
s_{i_j}\cdots s_{i_1}(\beta) - s_{i_{j-1}}\cdots s_{i_1}(\beta)
\]
is a positive multiple of $\alpha_{i_j}$.
\end{definition}

The root poset is graded by \textbf{depth}, the length of any reduced expression for a root, analogously to how the weak order is graded by length. Our first result shows that depth grows in a predictable way:

\begin{theorem}[Corollary \ref{lineardepth}]
\label{introlineardepth}
For a positive root $\alpha$, there exists an integer $t$ such that 
\[
\depth(\st_n(\alpha)) = tn + \depth(\alpha).
\]
\end{theorem}

\subsection{Downsets}

In section \ref{downsets}, we consider the \textbf{downset} generated by a positive root $\alpha$, the set of roots below $\alpha$ in the root poset, which we denote by $\downset \alpha$. Through the above interpretation of saturated chains, this encapsulates all reduced expressions for $\alpha$. For a fixed $\alpha$, we construct a finite structure which, for sufficiently large $n$, determines $\downset\st_n(\alpha)$. This gives a sense in which the root poset stabilizes. 

Specifically, for a root system $\Phi$ with diagram $G$, we define special subsets, called \textbf{stretching classes} (Definition \ref{stretchingclass}), of the set $\bigsqcup_n \st_n(\Phi)^+$ of positive roots for all stretches of $\Phi$. A stretching class consists of functions on the vertices of $\st_n(G)$ with:
\begin{itemize}
\item specified values off the stretched path and at some vertices on the ends of the stretched path, and
\item specified values occurring on the remaining vertices of the stretched path in a specified order, but each repeating any nonzero number of times.
\end{itemize}

Our main result is:

\begin{theorem}[Theorem \ref{stableposet}]
Let $\alpha$ be a positive root of $\Phi$. Then there exists a finite set of stretching classes for $G$ such that, for sufficiently large $n$, $\downset\st_n(\alpha)$ consists of the roots in $\st_n(\Phi)$ which lie in those stretching classes.
\end{theorem}

As a corollary, we deduce:

\begin{theorem}[Theorem \ref{polynomialgrowth}]
There is a polynomial $p(n)$ such that $|\downset\st_n(\alpha)| = p(n)$ for sufficiently large $n$. 
\end{theorem}

This generalizes the fact that the number of roots in the $A_n$ system is a quadratic polynomial in $n$.

\subsection{Shards}
\label{shardsintro}

Given a root system $\Phi$, the associated Coxeter group $W$ acts by reflections on a vector space $V^*$. 
For each root $\alpha$, there is a reflection $s_\alpha\in W$ which fixes a hyperplane $\alpha^\perp\subset V^*$. The arrangement of reflecting hyperplanes slices $V^*$ into regions. There is a natural choice of base region, which is bounded by the reflecting hyperplanes of the simple roots. This region is a fundamental domain for the action of $W$, and so by labeling it with the identity, we get a bijection between $W$ and the set of regions. This hyperplane arrangement is closely connected with the right weak order on $W$: two regions are adjacent along a hyperplane if and only if one of the associated elements of $W$ covers the other in weak order.

Studying lattice quotients of the weak order, Reading introduced a way of breaking the reflecting hyperplanes into convex subsets called \textbf{shards} \cite{reading1}\cite{reading2}. Shards govern lattice quotients in the following sense: given any quotient of the weak order, the identifications it makes between elements can be described in terms of merging regions by removing shards. They have since also been found to have representation-theoretic significance: given a simply laced Dynkin diagram, the shards of its Coxeter arrangement correspond to the bricks of the associated preprojective algebra \cite{thomas}.

Given a hyperplane arrangement $\AA$ with a distinguished base region $D$, along with two hyperplanes $H_1, H_2\in \AA$, we define the \textbf{rank 2 subarrangement} generated by $H_1$ and $H_2$ to consist of all hyperplanes of $\AA$ which contain $H_1\cap H_2$. One region of this subarrangement will contain $D$, and we call the two hyperplanes bounding this region the \textbf{basic hyperplanes}.

We then say that $H_1$ \textbf{cuts} $H_2$ if, in the rank 2 subarrangement they generate, $H_1$ is basic but $H_2$ is not. In any such pair, we call $H_1\cap H_2$ a \textbf{fracture} of $H_2$. Then the fractures of a hyperplane will slice it into pieces which we call \textbf{shards}.

Here, we fix a root $\alpha$ and consider the fractures of $\alpha^\perp$ as an arrangement within $\alpha^\perp$. This is a situation in which the type $A$ case is simple: letting $\alpha$ be the highest root of $A_n$, with value $1$ at every vertex, the arrangement of fractures is linearly isomorphic to the arrangement of coordinate hyperplanes in $\R^{n-1}$. In particular, the number of shards in $\alpha^\perp$ is $2^{n-1}$.

Inspired by this, we give a uniform description of the arrangement of fractures for any root system as we stretch a root $\alpha$. In the three results that follow, the constant $t$ is the depth growth rate from Theorem \ref{introlineardepth}.

\begin{theorem}[Theorem \ref{fractureform}]
Let $\alpha$ be a root of $\Phi$. Then there exists a vector space $V'$, linear forms $f_1,\ldots, f_s$ and $g_1,\ldots, g_t$ on $V'$, and an integer $e$ such that, for sufficiently large $n$, the arrangement of fractures in $\st_n(\alpha)^\perp$ is linearly isomorphic to the arrangement in $V'\times \R^{n - e}$ defined by the hyperplanes
\begin{align*}
f_i = 0 &\qquad 1\leq i\leq s \\
g_j - z_k = 0 &\qquad 1\leq j\leq t, 1\leq k\leq n-e
\end{align*}
\end{theorem}

Once we have this form for our arrangement, we get a uniform description of its characteristic polynomial (as defined in \cite{athanasiadis}). 

\begin{theorem}[Theorem \ref{charpoly}]
Let $\chi_n(q)$ be the characteristic polynomial of the arrangement of fractures of $\st_n(\alpha)$. Then there exist polynomials $p_0(q),\ldots, p_t(q)$ and an integer $e$ such that
\[
\chi_n(q) = \sum_{k=1}^t p_k(q)(q-k)^{n-e}
\]
for sufficiently large $n$.
\end{theorem}

Evaluating the characteristic polynomial at $-1$ gives the number of regions of an arrangement up to sign, so in particular, we get a uniform description of the number of shards:

\begin{corollary}[Corollary \ref{exponentialgrowth}]
Let $d$ be the number of vertices of $G$. Then for sufficiently large $n$, the number of shards of $\st_n(\alpha)^\perp$ is
\[
(-1)^{d - e - 1}\sum_{k=1}^t p_k(-1)(k+1)^{n-e}
\]
In particular, it is $O((t+1)^n)$.
\end{corollary}

In future research, we hope to explore how this perspective informs the aforementioned connection to representations of preprojective algebras.

\subsection{Acknowledgments} 

I would like to thank David Speyer for extensive advice, feedback, and encouragement.

Propositions \ref{fractureinductionstep} and \ref{fracturelist} originate in unpublished work of David Speyer and Hugh Thomas, and I thank them for allowing these results to appear here.

I was supported in part by NSF grants DMS-1840234 and DMS-1855135.

\section{Setup}

\subsection{Root systems and Coxeter groups}

Here we introduce the basic terminology and notation of root systems and Coxeter groups which we will use. For details, see Chapter 4 of \cite{bjornerbrenti}.

A \textbf{Coxeter group} $W$ is a group defined by generators and relations as
\[
\<s_1,\ldots, s_n\mid (s_is_j)^{m_{ij}} = 1\>
\]
for $m_{ij} \in \N \cup \{\infty\}$ such that $m_{ij} = m_{ji}$, $m_{ii} = 1$, and $m_{ij} \geq 2$ for $i\neq j$. (If $m_{ij} = \infty$, we omit that relation.) 

We store the data of a Coxeter group in its \textbf{Coxeter diagram} $G$, which is a graph with a vertex for each generator and an edge between vertices $i$ and $j$ if $m_{ij} \geq 3$, labeled by $m_{ij}$ if $m_{ij} \geq 4$.

A \textbf{Cartan matrix} for a Coxeter group is an $n\times n$ matrix $A$ such that $A_{ij} \leq 0$ for $i\neq j$, and
\[
\begin{aligned}
A_{ii} &= 2 && \\
A_{ij} &= 0 && m_{ij} = 2 \\
A_{ij}A_{ji} &= 4\cos^2 \frac{\pi}{m_{ij}} && 3\leq m_{ij} < \infty \\
A_{ij}A_{ji} &\geq 4 && m_{ij} = \infty	
\end{aligned}
\]
In this paper, we further assume that the $A_{ij}$ are integers, in which case the Cartan matrix is called \textbf{crystallographic}. In particular, this requires $m_{ij}\in \{1,2,3,4,6,\infty\}$.

Let $V$ be a real vector space with a basis $\{\alpha_1,\ldots, \alpha_n\}$ corresponding to the generators $s_i$. Then we define a bilinear pairing $(-,-):V\times V\to \R$ by $(\alpha_i, \alpha_j) = A_{ij}$. We let the generators act on $V$ by
\[
s_i(\beta) = \beta - (\alpha_i, \beta)\alpha_i.
\]
and this turns out to define a faithful representation of $W$. This action preserves the pairing, in that $(w\beta, w\gamma) = (\beta, \gamma)$. 

We call the $\alpha_i$ \textbf{simple roots}, and define the \textbf{root system} $\Phi$ to be the union of their $W$-orbits. Just as the simple roots are associated to the simple reflections, to any root $\alpha = w\alpha_i$ we associate the element $s_\alpha := ws_iw^{-1}$, which acts by
\[
ws_iw^{-1}(\beta) = \beta - (\alpha, \beta)\alpha.
\] 
Since each root is a unique linear combination of the simple roots, which correspond to the vertices of the Coxeter diagram, we will view roots as integer-valued functions on the vertices, and refer to their ``coefficients'' and ``values'' interchangeably. The coefficients of a root are either all positive or all negative, which partitions $\Phi$ into a set of \textbf{positive roots} $\Phi^+$ and \textbf{negative roots} $\Phi^-$. The negative of a root is a root, so we usually only need to consider positive roots.

We will want an explicit rule for applying reflections to functions on the diagram. Let $x$ be a vertex in the Coxeter diagram with edges to vertices $y_1,\ldots, y_k$. Let $\alpha$ be a root. Then $s_x(\alpha)$ will differ from $\alpha$ only in its value at $x$, which will be $\left(\sum_j -A_{xy_j}\alpha(y_j)\right) - \alpha(x)$.

In what follows, we will use the notation given in this section by default: 
\begin{itemize}
\item $W$ is a Coxeter group, with generators $s_i$;
\item $G$ is its Coxeter diagram;
\item $A$ is a Cartan matrix for $W$;
\item $\Phi$ is the associated root system, with simple roots $\alpha_i$.
\end{itemize}

\subsection{Reduced expressions and the root poset}

We will be interested in ways of obtaining a positive root by applying simple reflections to a simple root. 

\begin{definition}
An expression of the form
\[
\alpha = s_{y_n}\cdots s_{y_2}s_{y_1}(\alpha_{y_0})
\]
is a \textbf{reduced expression} for $\alpha$ if $n$ is minimal among all such expressions for $\alpha$.
\end{definition}

This is analogous to the notion of reduced words for elements of $W$. We also define the analogue of length:

\begin{definition}
The \textbf{depth} of a positive root $\alpha$ is the length of a reduced expression for $\alpha$, counting the simple root it starts with.
\end{definition}

We capture all the expressions for roots in a poset, analogous to the weak order on $W$:

\begin{definition}[{\cite[Definition 4.6.3]{bjornerbrenti}}]
Let $\alpha, \beta\in \Phi^+$. We say that $\alpha\leq \beta$ in the \textbf{root poset} if there exist simple reflections $s_1,\ldots, s_k$ such that:
\begin{itemize}
\item[(1)] $\beta = s_ks_{k-1}\cdots s_1(\alpha)$
\item[(2)] $\depth(s_is_{i-1}\cdots s_1(\alpha)) = \depth(\alpha) + i$ for all $1\leq i\leq k$.
\end{itemize}
\end{definition}

The root poset is graded by depth, and reduced expressions correspond to saturated chains based at simple roots.

We can also describe the root poset by its cover relations:

\begin{lemma}[{\cite[Lemma 4.6.4]{bjornerbrenti}}]
Let $\alpha_i$ be a simple root, $s_i$ the associated reflection, and $\beta$ an arbitrary positive root. Then $\beta \lessdot s_i(\beta)$ if $s_i(\beta) - \beta$ is a positive multiple of $\alpha_i$, and these are all the cover relations.
\end{lemma}

Figure \ref{a3rootposet} shows the root poset of $A_3$. It gives 4 different reduced expressions for the root with value $1$ at each vertex:
\[
s_3s_2(\alpha_1) = s_3s_1(\alpha_2) = s_1s_3(\alpha_2) = s_1s_2(\alpha_3)
\]

\begin{figure}
\centering
\begin{tikzpicture}
\node(100) at (-3, 0) {$100$};
\node(010) at (0, 0) {$010$};
\node(001) at (3, 0) {$001$};

\node(110) at (-1.5, 1.5) {$110$};
\node(011) at (1.5, 1.5) {$011$};

\node(111) at (0, 3) {$111$};

\draw (100) --node[above left]{$s_2$} (110);
\draw (010) --node[above right]{$s_1$} (110);
\draw (010) --node[above left]{$s_3$} (011);
\draw (001) --node[above right]{$s_2$} (011);
\draw (110) --node[above left]{$s_3$} (111);
\draw (011) --node[above right]{$s_1$} (111);
\end{tikzpicture}
\caption{The root poset of $A_3$. Each cover is labeled with the simple reflection that induces it.}
\label{a3rootposet}
\end{figure}
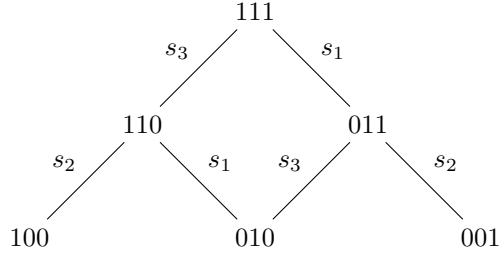

\subsection{Stretching}
\label{stretching}

Recall from Definition \ref{stretchdef} the definition of \textbf{elastic data} $(x, L_x, R_x)$ for a Coxeter diagram $G$ and the resulting stretched diagrams $\st_n(G)$. We construct a Cartan matrix and root system for $\st_n(G)$ by making the simple roots associated to the vertices on the stretched path pair symmetrically.

\begin{definition}
\label{stretchingrootsystems}
Let $A$ be a Cartan matrix for the Coxeter diagram $G$. Then the \textbf{$n$-stretched Cartan matrix}, $\st_n(A)$, has rows and columns indexed by the vertices of $\st_n(G)$, with
\[
\st_n(A)_{yz} = \begin{cases}
A_{yz} & y,z\notin\{x_0,\ldots, x_n\} \\
A_{xz} & (y = x_0 \text{ and } z\in L_x) \text{ or } (y = x_n \text{ and } z\in R_x) \\
A_{yx} & (z = x_0 \text{ and } y\in L_x) \text{ or } (z = x_n \text{ and } y\in R_x) \\
2 & y = z = x_i \\
-1 & y = x_i, z = x_{i\pm 1}\\
0 & \text{otherwise}
\end{cases}
\]
If $\Phi$ is the root system associated to $A$, then the \textbf{$n$-stretched root system}, $\st_n(\Phi)$, is the root system associated to $\st_n(A)$. 
\end{definition} 

In what follows, $x$, $L_x$, and $R_x$ will refer to elastic data by default.

Now for any integer-valued function $\alpha$ on the vertices of $G$, we recall from Definition \ref{rootstretchdef} the definition of $\st_n(\alpha)$.

\begin{proposition}
\label{stretchingroots}
Let $\alpha\in \Phi$. Then $\st_n(\alpha) \in \st_n(\Phi)$.
\end{proposition}

\begin{proof}
It will suffice to assume $\alpha$ is positive. We proceed by induction on $\depth(\alpha)$. First consider the base case that $\alpha$ is simple: either $\st_n(\alpha)$ is also simple, or it has value 1 on the stretched path and 0 elsewhere, and this is straightforward to obtain by reflections of a simple root.

Now consider any positive root $\alpha$. If it is possible to reflect at a vertex other than $x$ and decrease the value there, obtaining a root $\alpha'$, then we can perform the same operation to $\st_n(\alpha)$ and get $\st_n(\alpha')$. By the induction hypothesis, $\st_n(\alpha')$ is a root, so $\st_n(\alpha)$ is too. 

Otherwise, we can reflect at $x$ and decrease the value there to get $\alpha' := s_x(\alpha)$. Then
\begin{align*}
\st_n(\alpha') &= s_{x_0}s_{x_1}\cdots s_{x_n} s_{x_{n-1}} \cdots s_{x_1}s_{x_0} (\st_n(\alpha)) \\
&= s_{x_n}s_{x_{n-1}}\cdots s_{x_0} s_{x_{1}} \cdots s_{x_{n-1}}s_{x_n} (\st_n(\alpha)).
\end{align*}
By the induction hypothesis, $\st_n(\alpha')$ is a root, so $\st_n(\alpha)$ is too.
\end{proof}

As an aside, we note that a reverse version of Proposition \ref{stretchingroots} also holds: roots with repeated coefficients can be squished to give roots of a smaller diagram.

\begin{proposition}
\label{squishing}
Let $\alpha$ be any integer-valued function on the vertices of $G$ such that $\st_1(\alpha)$ is a root of $\st_1(\Phi)$. Then $\alpha$ is a root of $\Phi$. 
\end{proposition}

\begin{proof}

We assume without loss of generality that $\st_1(\alpha)$ is positive. Let $\wt{\alpha} := \st_1(\alpha)$, and define roots in $\st_1(\Phi)$ by
\[
\wt\alpha_y := \begin{cases}
\alpha_y & y\neq x \\
\alpha_{x_0} + \alpha_{x_1} & y = x
\end{cases}
\]
Then $\wt{\alpha}$ is a nonnegative linear combination of these.

We proceed by induction on $\depth(\wt\alpha)$. The base case is when $\wt\alpha = \wt\alpha_y$, which is trivial. Now suppose $\wt\alpha$ is different from these. We have $(\wt\alpha, \wt\alpha) = 2 > 0$, so at least one of the pairings $(\wt\alpha_y, \wt\alpha)$ is positive. 

If $(\wt\alpha_y, \wt\alpha) > 0$ for some $y\neq x$, then $s_y(\wt\alpha) - \wt\alpha = -(\wt\alpha_y, \wt\alpha)\wt\alpha_y$ is a negative multiple of $\wt\alpha_y$, and so $s_y(\wt\alpha) < \wt\alpha$ in the root poset. We still have $s_y(\wt\alpha)(x_0) = s_y(\wt\alpha)(x_1)$, and so by the induction hypothesis there exists a root $\alpha'\in \Phi$ such that $\st_1(\alpha') = s_y(\wt\alpha)$. But then $s_y(\alpha') = \alpha$, so $\alpha$ is a root.

On the other hand, if $(\wt\alpha_x, \wt\alpha) > 0	$, then $s_{x_0}s_{x_1}s_{x_0}(\wt\alpha) = \wt\alpha - (\wt\alpha_x, \wt\alpha)\wt\alpha_x$ has two coefficients which are smaller than those of $\wt\alpha$, so in applying $s_{x_0}s_{x_1}s_{x_0}$ we must have gone down in the root poset at least twice, implying $\depth(s_{x_0}s_{x_1}s_{x_0}(\wt\alpha)) < \depth(\wt\alpha)$. As above, by the induction hypothesis there is some $\alpha'\in \Phi$ such that $\st_1(\alpha') = s_{x_0}s_{x_1}s_{x_0}(\wt\alpha)$, and direct computation shows that $s_x(\alpha') = \alpha$, implying $\alpha$ is also a root.
\end{proof}

\subsection{Stretching and reduced expressions}
\label{stretchingandexpressions}

Iterating the process in Proposition \ref{stretchingroots} gives an expression for $\st_n(\alpha)$ in terms of simple reflections applied to a simple root, but it may not be a reduced expression, because two roots in a cover relation may no longer be comparable once stretched. We examine when this happens, and obtain a result on the depth of stretched roots in the process. 

Consider a positive root $\alpha$ such that $s_x(\alpha) < \alpha$. We examine just the coefficients at $x$ and its neighbors:
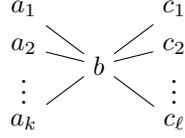
\begin{center}
\begin{tikzpicture}
\node(1) at (-1, 0.75) {$a_1$};
\node(2) at (-1, 0.25) {$a_2$};
\node(dots) at (-1, -0.25) {$\vdots$};
\node(3) at (-1, -0.75) {$a_k$};
\node(4) at (0, 0) {$b$};
\node(5) at (1, 0.75) {$c_1$};
\node(6) at (1, 0.25) {$c_2$};
\node(dots2) at (1, -0.25) {$\vdots$};
\node(7) at (1, -0.75) {$c_\ell$};

\draw (1) -- (4);
\draw (2) -- (4);
\draw (3) -- (4);
\draw (4) -- (5);
\draw (4) -- (6);
\draw (4) -- (7);
\end{tikzpicture}
\end{center}

Let $y_1,\ldots, y_k$ be the left neighbors of $x$ and let $z_1,\ldots, z_\ell$ be the right neighbors. Let $S_L := \sum_i -A_{xy_i}a_i$ and $S_R := \sum_j -A_{xz_j} c_j$. Then to assume $s_x(\alpha) < \alpha$ means $S_L + S_R - b < b$. In particular, one of $S_L$ and $S_R$, without loss of generality $S_L$, is less than $b$. Then the situation splits into three cases depending on $S_R$:
\begin{itemize}
\item[(1)] $S_R = b$.
\item[(2)] $S_R < b$.
\item[(3)] $S_R > b$.
\end{itemize}
This trichotomy classifies the different outcomes of stretching the roots in a cover relation:
\begin{lemma}
\label{stretchedcovers}
\begin{itemize}
\item[(1)] In Case 1 above, $\st_n(s_x(\alpha)) < \st_n(\alpha)$, and 
\[
\on{depth}(\st_n(\alpha)) = \on{depth}(\st_n(s_x(\alpha))) + (n+1).
\]
\item[(2)] In Case 2, $\st_n(s_x(\alpha)) < \st_n(\alpha)$, and
\[
\on{depth}(\st_n(\alpha)) = \on{depth}(\st_n(s_x(\alpha))) + (2n+1).
\]
\item[(3)] In Case 3, $\st_n(s_x(\alpha))$ and $\st_n(\alpha)$ are incomparable, and
\[
\on{depth}(\st_n(\alpha)) = \on{depth}(\st_n(s_x(\alpha))) + 1.
\]
\end{itemize}
\end{lemma}

\begin{proof}
We know from the proof of Proposition \ref{stretchingroots} that 
\[
\st_n(s_x(\alpha)) = s_{x_0}s_{x_1}\cdots s_{x_n}s_{x_{n-1}}\cdots s_{x_1}s_{x_0}(\st_n(\alpha))
\]
We thus check, in each case of the trichotomy, whether each of these simple reflections steps up or down in the root poset, and use the fact that the poset is graded by depth. Let $b' := S_L + S_R - b$ be the value at $x$ in $s_x(\alpha)$. Then Figure \ref{backandforth} shows the result of applying each reflection in turn.

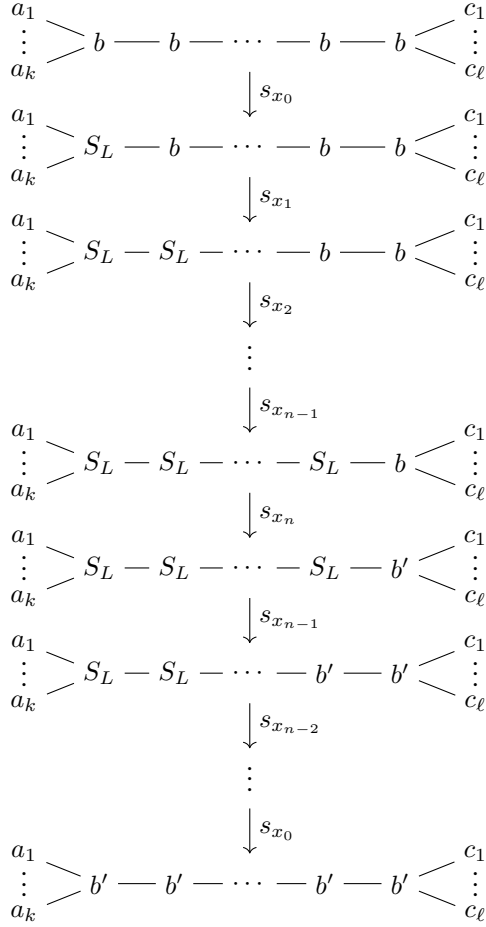
\begin{figure}
\centering
\begin{tikzpicture}
\node(1a) at (-1, 0.4) {$a_1$};
\node(dotsa) at (-1, 0.1) {$\vdots$};
\node(ka) at (-1, -0.4) {$a_k$};
\node(40a) at (0, 0) {$b$};
\node(41a) at (1, 0) {$b$};
\node(4dotsa) at (2, 0) {$\cdots$};
\node(4n-1a) at (3, 0) {$b$};
\node(4na) at (4, 0) {$b$};
\node(5a) at (5, 0.4) {$c_1$};
\node(dots2a) at (5, 0.1) {$\vdots$};
\node(la) at (5, -0.4) {$c_\ell$};

\draw (1a) -- (40a);
\draw (ka) -- (40a);
\draw (40a) -- (41a);
\draw (41a) -- (4dotsa);
\draw (4dotsa) -- (4n-1a);
\draw (4n-1a) -- (4na);
\draw (4na) -- (5a);
\draw (4na) -- (la);

\draw[->] (2, -0.4) -- node[right]{$s_{x_0}$} (2, -1.0);

\node(1b) at (-1, -1.0) {$a_1$};
\node(dotsb) at (-1, -1.3) {$\vdots$};
\node(kb) at (-1, -1.8) {$a_k$};
\node(40b) at (0, -1.4) {$S_L$};
\node(41b) at (1, -1.4) {$b$};
\node(4dotsb) at (2, -1.4) {$\cdots$};
\node(4n-1b) at (3, -1.4) {$b$};
\node(4nb) at (4, -1.4) {$b$};
\node(5b) at (5, -1.0) {$c_1$};
\node(dots2b) at (5, -1.3) {$\vdots$};
\node(lb) at (5, -1.8) {$c_\ell$};

\draw (1b) -- (40b);
\draw (kb) -- (40b);
\draw (40b) -- (41b);
\draw (41b) -- (4dotsb);
\draw (4dotsb) -- (4n-1b);
\draw (4n-1b) -- (4nb);
\draw (4nb) -- (5b);
\draw (4nb) -- (lb);

\draw[->] (2, -1.8) -- node[right]{$s_{x_1}$} (2, -2.4);

\node(1c) at (-1, -2.4) {$a_1$};
\node(dotsc) at (-1, -2.7) {$\vdots$};
\node(kc) at (-1, -3.2) {$a_k$};
\node(40c) at (0, -2.8) {$S_L$};
\node(41c) at (1, -2.8) {$S_L$};
\node(4dotsc) at (2, -2.8) {$\cdots$};
\node(4n-1c) at (3, -2.8) {$b$};
\node(4nc) at (4, -2.8) {$b$};
\node(5c) at (5, -2.4) {$c_1$};
\node(dots2c) at (5, -2.7) {$\vdots$};
\node(lc) at (5, -3.2) {$c_\ell$};

\draw (1c) -- (40c);
\draw (kc) -- (40c);
\draw (40c) -- (41c);
\draw (41c) -- (4dotsc);
\draw (4dotsc) -- (4n-1c);
\draw (4n-1c) -- (4nc);
\draw (4nc) -- (5c);
\draw (4nc) -- (lc);

\draw[->] (2, -3.2) -- node[right]{$s_{x_2}$} (2, -3.8);

\node at (2, -4.1) {$\vdots$};

\draw[->] (2, -4.6) -- node[right]{$s_{x_{n-1}}$} (2, -5.2);

\node(1d) at (-1, -5.2) {$a_1$};
\node(dotsd) at (-1, -5.5) {$\vdots$};
\node(kd) at (-1, -6.0) {$a_k$};
\node(40d) at (0, -5.6) {$S_L$};
\node(41d) at (1, -5.6) {$S_L$};
\node(4dotsd) at (2, -5.6) {$\cdots$};
\node(4n-1d) at (3, -5.6) {$S_L$};
\node(4nd) at (4, -5.6) {$b$};
\node(5d) at (5, -5.2) {$c_1$};
\node(dots2d) at (5, -5.5) {$\vdots$};
\node(ld) at (5, -6.0) {$c_\ell$};

\draw (1d) -- (40d);
\draw (kd) -- (40d);
\draw (40d) -- (41d);
\draw (41d) -- (4dotsd);
\draw (4dotsd) -- (4n-1d);
\draw (4n-1d) -- (4nd);
\draw (4nd) -- (5d);
\draw (4nd) -- (ld);

\draw[->] (2, -6.0) -- node[right]{$s_{x_n}$} (2, -6.6);

\node(1e) at (-1, -6.6) {$a_1$};
\node(dotse) at (-1, -6.9) {$\vdots$};
\node(ke) at (-1, -7.4) {$a_k$};
\node(40e) at (0, -7.0) {$S_L$};
\node(41e) at (1, -7.0) {$S_L$};
\node(4dotse) at (2, -7.0) {$\cdots$};
\node(4n-1e) at (3, -7.0) {$S_L$};
\node(4ne) at (4, -7.0) {$b'$};
\node(5e) at (5, -6.6) {$c_1$};
\node(dots2e) at (5, -6.9) {$\vdots$};
\node(le) at (5, -7.4) {$c_\ell$};

\draw (1e) -- (40e);
\draw (ke) -- (40e);
\draw (40e) -- (41e);
\draw (41e) -- (4dotse);
\draw (4dotse) -- (4n-1e);
\draw (4n-1e) -- (4ne);
\draw (4ne) -- (5e);
\draw (4ne) -- (le);

\draw[->] (2, -7.4) -- node[right]{$s_{x_{n-1}}$} (2, -8.0);

\node(1f) at (-1, -8.0) {$a_1$};
\node(dotsf) at (-1, -8.3) {$\vdots$};
\node(kf) at (-1, -8.8) {$a_k$};
\node(40f) at (0, -8.4) {$S_L$};
\node(41f) at (1, -8.4) {$S_L$};
\node(4dotsf) at (2, -8.4) {$\cdots$};
\node(4n-1f) at (3, -8.4) {$b'$};
\node(4nf) at (4, -8.4) {$b'$};
\node(5f) at (5, -8.0) {$c_1$};
\node(dots2f) at (5, -8.3) {$\vdots$};
\node(lf) at (5, -8.8) {$c_\ell$};

\draw (1f) -- (40f);
\draw (kf) -- (40f);
\draw (40f) -- (41f);
\draw (41f) -- (4dotsf);
\draw (4dotsf) -- (4n-1f);
\draw (4n-1f) -- (4nf);
\draw (4nf) -- (5f);
\draw (4nf) -- (lf);

\draw[->] (2, -8.8) -- node[right]{$s_{x_{n-2}}$} (2, -9.4);

\node at (2, -9.7) {$\vdots$};

\draw[->] (2, -10.2) -- node[right]{$s_{x_0}$} (2, -10.8);

\node(1g) at (-1, -10.8) {$a_1$};
\node(dotsg) at (-1, -11.1) {$\vdots$};
\node(kg) at (-1, -11.6) {$a_k$};
\node(40g) at (0, -11.2) {$b'$};
\node(41g) at (1, -11.2) {$b'$};
\node(4dotsg) at (2, -11.2) {$\cdots$};
\node(4n-1g) at (3, -11.2) {$b'$};
\node(4ng) at (4, -11.2) {$b'$};
\node(5g) at (5, -10.8) {$c_1$};
\node(dots2g) at (5, -11.1) {$\vdots$};
\node(lg) at (5, -11.6) {$c_\ell$};

\draw (1g) -- (40g);
\draw (kg) -- (40g);
\draw (40g) -- (41g);
\draw (41g) -- (4dotsg);
\draw (4dotsg) -- (4n-1g);
\draw (4n-1g) -- (4ng);
\draw (4ng) -- (5g);
\draw (4ng) -- (lg);
\end{tikzpicture}
\caption{The sequence of roots appearing at each step of the expression $\st_n(s_x(\alpha)) = s_{x_0}s_{x_1}\cdots s_{x_n}s_{x_{n-1}}\cdots s_{x_1} s_{x_0}(\st_n(\alpha))$.}
\label{backandforth}
\end{figure}

\begin{itemize}
\item[(1)] If $S_R = b$, then $b' = S_L$. In particular, after we apply $s_{x_n}$ halfway through the chain, we have already reached $\st_n(s_x(\alpha))$ after $n+1$ steps. Since $S_L < b$, at each of those steps we go down in the root poset, so $\st_n(s_x(\alpha)) < \st_n(\alpha)$ and $\depth(\st_n(s_x(\alpha))) = \depth(\st_n(\alpha)) - (n+1)$.

\item[(2)] If $S_R < b$, then $b' < S_L < b$. Thus at each of the $2n + 1$ steps in Figure \ref{backandforth}, a coefficient strictly decreases. Thus $\st_n(s_x(\alpha)) < \st_n(\alpha)$ and $\depth(\st_n(s_x(\alpha))) = \depth(\st_n(\alpha)) - (2n+1)$.

\item[(3)] If $S_R > b$, then $b' > S_L < b$. Thus the first $n+1$ steps in Figure \ref{backandforth} move down in the root poset, while the remaining $n$ move back up. Thus $\depth(\st_n(s_x(\alpha))) = \depth(\st_n(\alpha)) - 1$. In particular, if $\st_n(s_x(\alpha))$ were comparable to $\st_n(\alpha)$, it would be covered by $\st_n(\alpha)$, but since they differ in more than one coefficient this is not possible.
\end{itemize}
\end{proof}

This gives us our first numerical result on stretching a root:

\begin{corollary}
\label{lineardepth}
For any positive root $\alpha$, there exists an integer $t$ such that $\depth(\st_n(\alpha)) = tn + \depth(\alpha)$.
\end{corollary}

\begin{proof}
By induction on $\on{depth}(\alpha)$. For any simple root based away from vertex $x$, the stretched depth is 1, while for the simple root $\alpha_x$ it is $n+1$. 

Then suppose we have a cover $s_z(\alpha) \lessdot \alpha$ in the root poset, such that $\on{depth}(\st_n(s_z(\alpha))) = t'n + \depth(s_z(\alpha))$. If $z\neq x$, then $\st_n(s_z(\alpha)) \lessdot \st_n(\alpha)$ is still a cover, and so 
\[
\depth(\st_n(\alpha)) = t'n + \depth(s_z(\alpha)) + 1 = t'n + \depth(\alpha).
\]

If $z = x$, then Lemma \ref{stretchedcovers} implies
\[
\depth(\st_n(\alpha)) = t'n + \depth(s_x(\alpha)) + cn + 1 = (t' + c)n + \depth(\alpha),
\] 
where $c = 0$, $1$, or $2$.
\end{proof}

\begin{definition}
\label{depthgrowthrate}
The \textbf{depth growth rate} of $\alpha$ is this integer $t$. 
\end{definition}

We end this section by looking in more detail at examples of stretching the roots in a cover relation.

In case (1), since $\st_n(\alpha)$ and $\st_n(s_x(\alpha))$ are comparable, we can consider the interval between them. In this case it is just the chain forming the top half of Figure \ref{backandforth}: each reflection in that chain is the only one we can make while decreasing a coefficient on the stretched path.

In case (2), on the other hand, there are other ways of getting from $\st_n(\alpha)$ down to $\st_n(s_x(\alpha))$. Figure \ref{stretchedinterval} shows a cover relation exhibiting case (2), together with the interval between their $3$-stretched versions. This reveals a bit of type $A$ behavior: the interval consists of two copies of the $A_n$ root poset, one inverted.

In case (3), although $\st_n(\alpha)$ and $\st_n(s_x(\alpha))$ are incomparable, we can situate them in what is essentially a sideways version of the interval in Figure \ref{stretchedinterval}, accounting for the fact that reflections which go down in the root poset in case (2) may go up in case (3). This is illustrated in Figure \ref{stretchednoninterval}.

In all cases, we see the effect of stretching on depth stated in Lemma \ref{stretchedcovers}.
\begin{figure}
\centering

\begin{minipage}{0.4\textwidth}
\centering
\begin{tikzpicture}
\node(10) at (-1, 0.5) {$1$};
\node(20) at (-1, 0) {$1$};
\node(30) at (-1, -0.5) {$1$};
\node(40) at (0, 0) {$4$};
\node(50) at (1, 0.5) {$1$};
\node(60) at (1, -0.5) {$1$};

\draw (10) -- (40);
\draw (20) -- (40);
\draw (30) -- (40);
\draw (40) -- (50);
\draw (40) -- (60);

\draw[->] (0, -0.5) -- (0, -1.5);

\node(11) at (-1, -1.5) {$1$};
\node(21) at (-1, -2) {$1$};
\node(31) at (-1, -2.5) {$1$};
\node(41) at (0, -2) {$1$};
\node(51) at (1, -1.5) {$1$};
\node(61) at (1, -2.5) {$1$};

\draw (11) -- (41);
\draw (21) -- (41);
\draw (31) -- (41);
\draw (41) -- (51);
\draw (41) -- (61);
\end{tikzpicture}
\end{minipage}
\begin{minipage}{0.4\textwidth}
\centering

\begin{tikzpicture}[scale=0.75]
\node[draw=black](4444) at (0, 4) {$4444$};
\node(3444) at (-1, 3) {$3444$};
\node(4442) at (1, 3) {$4442$};
\node(3344) at (-2, 2) {$3344$};
\node(3442) at (0, 2) {$3442$};
\node(4422) at (2, 2) {$4422$};
\node(3334) at (-3, 1) {$3334$};
\node(3342) at (-1, 1) {$3342$};
\node(3422) at (1, 1) {$3422$};
\node(4222) at (3, 1) {$4222$};
\node(3331) at (-3, 0) {$3331$};
\node(3312) at (-1, 0) {$3312$};
\node(3122) at (1, 0) {$3122$};
\node(1222) at (3, 0) {$1222$}; 
\node(3311) at (-2, -1) {$3311$};
\node(3112) at (0, -1) {$3112$};
\node(1122) at (2, -1) {$1122$};
\node(3111) at (-1, -2) {$3111$};
\node(1112) at (1, -2) {$1112$};
\node[draw=black](1111) at (0, -3) {$1111$};

\draw (4444) -- (3444) -- (3344) -- (3334) -- (3331) -- (3311) -- (3111) -- (1111);
\draw (4442) -- (3442) -- (3342) -- (3312) -- (3112) -- (1112);
\draw (4422) -- (3422) -- (3122) -- (1122) ;
\draw (4222) -- (1222);
\draw (4444) -- (4442) -- (4422) -- (4222);
\draw (1222) -- (1122) -- (1112) -- (1111);
\draw (3444) -- (3442) -- (3422);
\draw (3122) -- (3112) -- (3111);
\draw (3344) -- (3342);
\draw (3312) -- (3311);
\end{tikzpicture}
\end{minipage}
\caption{A cover exhibiting case (2) and the interval between the $3$-stretched roots. To save space, roots on the right are represented by their values on the stretched path.}
\label{stretchedinterval}
\end{figure}

\begin{figure}
\centering

\begin{minipage}{0.4\textwidth}
\centering
\begin{tikzpicture}
\node(10) at (-1, 0.5) {$1$};
\node(20) at (-1, 0) {$1$};
\node(30) at (-1, -0.5) {$1$};
\node(40) at (0, 0) {$5$};
\node(50) at (1, 0.5) {$3$};
\node(60) at (1, -0.5) {$3$};

\draw (10) -- (40);
\draw (20) -- (40);
\draw (30) -- (40);
\draw (40) -- (50);
\draw (40) -- (60);

\draw[->] (0, -0.5) -- (0, -1.5);

\node(11) at (-1, -1.5) {$1$};
\node(21) at (-1, -2) {$1$};
\node(31) at (-1, -2.5) {$1$};
\node(41) at (0, -2) {$4$};
\node(51) at (1, -1.5) {$3$};
\node(61) at (1, -2.5) {$3$};

\draw (11) -- (41);
\draw (21) -- (41);
\draw (31) -- (41);
\draw (41) -- (51);
\draw (41) -- (61);
\end{tikzpicture}
\end{minipage}
\begin{minipage}{0.4\textwidth}
\centering
\begin{tikzpicture}[scale=0.75]
\node[draw=black](5555) at (-3, 1) {$5555$};
\node(5556) at (-2, 2) {$5556$};
\node(3555) at (-2, 0) {$3555$};
\node(5566) at (-1, 3) {$5566$};
\node(3556) at (-1, 1) {$3556$};
\node(3355) at (-1, -1) {$3355$};
\node(5666) at (0, 4) {$5666$};
\node(3566) at (0, 2) {$3566$};
\node(3356) at (0, 0) {$3356$};
\node(3335) at (0, -2) {$3335$};
\node(4666) at (1, 3) {$4666$};
\node(3466) at (1, 1) {$3466$};
\node(3346) at (1, -1) {$3346$};
\node(3334) at (1, -3) {$3334$};
\node(4466) at (2, 2) {$4466$};
\node(3446) at (2, 0) {$3446$};
\node(3344) at (2, -2) {$3344$};
\node(4446) at (3, 1) {$4446$};
\node(3444) at (3, -1) {$3444$};
\node[draw=black](4444) at (4, 0) {$4444$};

\draw (5555) -- (5556) -- (5566) -- (5666) -- (4666) -- (4466) -- (4446) -- (4444);
\draw (3555) -- (3556) -- (3566) -- (3466) -- (3446) -- (3444);
\draw (3355) -- (3356) -- (3346) -- (3344);
\draw (3335) -- (3334);
\draw (5555) -- (3555) -- (3355) -- (3335);
\draw (5556) -- (3556) -- (3356);
\draw (5566) -- (3566);
\draw (4466) -- (3466);
\draw (4446) -- (3446) -- (3346);
\draw (4444) -- (3444) -- (3344) -- (3334);
\end{tikzpicture}
\end{minipage}

\caption{A cover exhibiting case (3) and the analogue of the interval shown in Figure \ref{stretchedinterval}.}
\label{stretchednoninterval}
\end{figure}
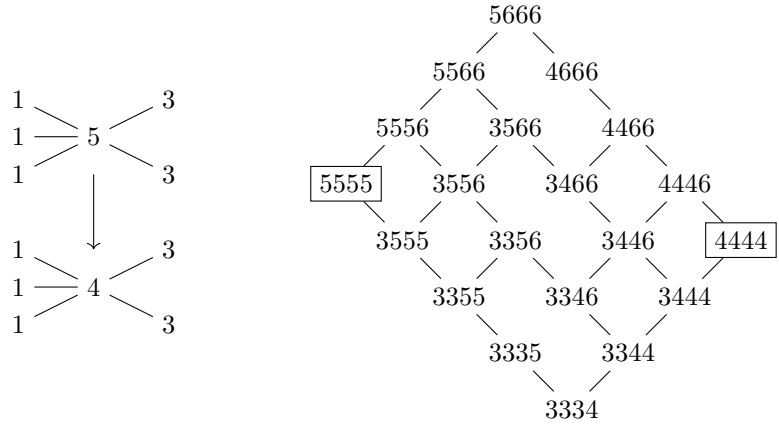

\section{Downsets in stretched root posets}
\label{downsets}
For a positive root $\alpha$, the \textbf{downset generated by $\alpha$} is
\[
\downset\alpha := \{\beta\in \Phi^+\mid \beta \leq \alpha\}
\]. 
This section will prove the following result:

\begin{theorem}
Let $\alpha$ be a positive root. Then there is a polynomial $p(n)$ such that $|\downset\st_n(\alpha)| = p(n)$ for sufficiently large $n$. 
\end{theorem}

We show this using a stability result, by constructing a single finite structure which gives $\downset\st_n(\alpha)$ for all sufficiently large $n$.

To motivate the kind of stability we will use, consider the finite-type case of $D_n$. In this case, the root poset has a unique maximal root, obtained by stretching the maximal root for $D_4$:
\begin{center}
\begin{tikzpicture}
\node(1) at (-1, 0.3) {$1$};
\node(2) at (-1, -0.3) {$1$};
\node(3) at (0, 0) {$2$};
\node(dots) at (1, 0) {$\cdots$};
\node(n-1) at (2, 0) {$2$};
\node(n) at (3, 0) {$1$};

\draw (1) -- (3) -- (dots) -- (n-1) -- (n);
\draw (2) -- (3);
\end{tikzpicture}
\end{center}
Thus the entire set of positive roots for $D_n$ is of the form $\downset\st_{n-4}(\alpha)$. 

The roots of $D_n$ for all $n$ fit a finite list of patterns. For example, any function of the form
\begin{center}
\begin{tikzpicture}
\node(1) at (-1, 0.3) {$1$};
\node(2) at (-1, -0.3) {$1$};
\node(3) at (0, 0) {$2$};
\node(dots1) at (1, 0) {$\cdots$};
\node(k) at (2, 0) {$2$};
\node(k+1) at (3, 0) {$1$};
\node(dots2) at (4, 0) {$\cdots$};
\node(n-1) at (5, 0) {$1$};
\node(n) at (6, 0) {$1$};

\draw (1) -- (3) -- (dots1) -- (k) -- (k+1) -- (dots2) -- (n-1) -- (n);
\draw (2) -- (3);
\end{tikzpicture}
\end{center}
is a root. We can compactly describe the roots of this form with the notation
\begin{center}
\begin{tikzpicture}
\node(1) at (-1, 0.3) {$1$};
\node(2) at (-1, -0.3) {$1$};
\node(3) at (0, 0) {$2^*$};
\node(4) at (1, 0) {$1^*$};
\node(5) at (2, 0) {$1$};

\draw (1) -- (3) -- (4) -- (5);
\draw (2) -- (3);
\end{tikzpicture}
\end{center}
using an asterisk on a coefficient to mean that it can repeat any nonzero number of times. (We distinguish the rightmost vertex because it is not part of the stretched path.) Then we can write down a list of expressions like this which describe the roots of every $D_n$, shown in Figure \ref{typeDclasses}.

\begin{figure}[h]
\centering

\begin{tabular}{ccc}
\begin{tikzpicture}
\node (1) at (-1, 0.3) {$1$};
\node (2) at (-1, -0.3) {$1$};
\node (3) at (0, 0) {$2^*$};
\node (4) at (1, 0) {$1$};
\draw (1) -- (3) -- (4);
\draw (2) -- (3);
\end{tikzpicture} & 
\begin{tikzpicture}
\node (1) at (-1, 0.3) {$1$};
\node (2) at (-1, -0.3) {$1$};
\node (3) at (0, 0) {$2^*$};
\node (4) at (1, 0) {$1^*$};
\node (5) at (2, 0) {$0/1$};
\draw (1) -- (3) -- (4) -- (5);
\draw (2) -- (3);
\end{tikzpicture} &
\begin{tikzpicture}
\node (1) at (-1, 0.3) {$0/1$};
\node (2) at (-1, -0.3) {$0/1$};
\node (3) at (0, 0) {$1^*$};
\node (4) at (1, 0) {$0/1$};
\draw (1) -- (3) -- (4);
\draw (2) -- (3);
\end{tikzpicture} \\

\begin{tikzpicture}
\node (1) at (-1, 0.3) {$0/1$};
\node (2) at (-1, -0.3) {$0/1$};
\node (3) at (0, 0) {$1^*$};
\node (4) at (1, 0) {$0^*$};
\node (5) at (2, 0) {$0$};
\draw (1) -- (3) -- (4) -- (5);
\draw (2) -- (3);
\end{tikzpicture} & 
\begin{tikzpicture}
\node (1) at (-1, 0.3) {$0$};
\node (2) at (-1, -0.3) {$0$};
\node (3) at (0, 0) {$0^*$};
\node (4) at (1, 0) {$1^*$};
\node (5) at (2, 0) {$0/1$};
\draw (1) -- (3) -- (4) -- (5);
\draw (2) -- (3);
\end{tikzpicture} & 
\begin{tikzpicture}
\node (1) at (-1, 0.3) {$0$};
\node (2) at (-1, -0.3) {$0$};
\node (3) at (0, 0) {$0^*$};
\node (4) at (1, 0) {$1$};
\draw (1) -- (3) -- (4);
\draw (2) -- (3);
\end{tikzpicture} \\

\begin{tikzpicture}
\node (1) at (-1, 0.3) {$0$};
\node (2) at (-1, -0.3) {$0$};
\node (3) at (0, 0) {$0^*$};
\node (4) at (1, 0) {$1^*$};
\node (5) at (2, 0) {$0^*$};
\node (6) at (3, 0) {$0$};
\draw (1) -- (3) -- (4) -- (5) -- (6);
\draw (2) -- (3);
\end{tikzpicture} &
\begin{tikzpicture}
\node (1) at (-1, 0.3) {$1$};
\node (2) at (-1, -0.3) {$0$};
\node (3) at (0, 0) {$0^*$};
\node (4) at (1, 0) {$0$};
\draw (1) -- (3) -- (4);
\draw (2) -- (3);
\end{tikzpicture} &
\begin{tikzpicture}
\node (1) at (-1, 0.3) {$0$};
\node (2) at (-1, -0.3) {$1$};
\node (3) at (0, 0) {$0^*$};
\node (4) at (1, 0) {$0$};
\draw (1) -- (3) -- (4);
\draw (2) -- (3);
\end{tikzpicture}
\end{tabular}

\caption{The roots of $D_n$ are precisely the functions which fit these patterns.}
\label{typeDclasses}
\end{figure}

In general, however, describing downsets is not quite as simple as allowing values to repeat freely on the stretched path. Consider the root used in Figure \ref{stretchednoninterval}:
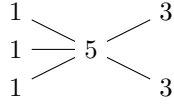
\begin{center}
\begin{tikzpicture}
\node(1) at (-1, 0.5) {$1$};
\node(2) at (-1, 0) {$1$};
\node(3) at (-1, -0.5) {$1$};
\node(4) at (0, 0) {$5$};
\node(5) at (1, 0.5) {$3$};
\node(6) at (1, -0.5) {$3$};

\draw (1) -- (4);
\draw (2) -- (4);
\draw (3) -- (4);
\draw (4) -- (5);
\draw (4) -- (6);
\end{tikzpicture}
\end{center}
The downsets generated by its stretches contain roots of the form
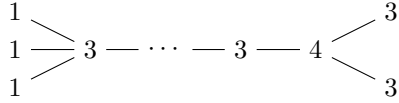
\begin{center}
\begin{tikzpicture}
\node(1) at (-1, 0.5) {$1$};
\node(2) at (-1, 0) {$1$};
\node(3) at (-1, -0.5) {$1$};
\node(4) at (0, 0) {$3$};
\node(5) at (1, 0) {$\cdots$};
\node(6) at (2, 0) {$3$};
\node(7) at (3, 0) {$4$};
\node(8) at (4, 0.5) {$3$};
\node(9) at (4, -0.5) {$3$};

\draw (1) -- (4) -- (5) -- (6) -- (7) -- (8);
\draw (7) -- (9);
\draw (2) -- (4);
\draw (3) -- (4);
\end{tikzpicture}
\end{center}
but not those of the form
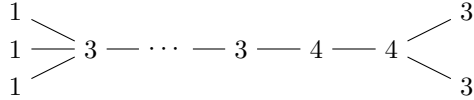
\begin{center}
\begin{tikzpicture}
\node(1) at (-1, 0.5) {$1$};
\node(2) at (-1, 0) {$1$};
\node(3) at (-1, -0.5) {$1$};
\node(4) at (0, 0) {$3$};
\node(5) at (1, 0) {$\cdots$};
\node(6) at (2, 0) {$3$};
\node(7) at (3, 0) {$4$};
\node(8) at (4, 0) {$4$};
\node(9) at (5, 0.5) {$3$};
\node(10) at (5, -0.5) {$3$};

\draw (1) -- (4) -- (5) -- (6) -- (7) -- (8) -- (9);
\draw (8) -- (10);
\draw (2) -- (4);
\draw (3) -- (4);
\end{tikzpicture}
\end{center}
as suggested by the poset in Figure \ref{stretchednoninterval}.

Thus, to describe downsets as in Figure \ref{typeDclasses}, we need our patterns to allow for some values at the ends of the stretched path to not repeat. 

\begin{definition}
\label{stretchingclass}
Let $\ol{\beta}$ be a integer-valued function on the vertices of some stretch of $G$, together with a marking of some consecutive vertices on the stretched path by asterisks, subject to the constraint that no adjacent asterisked values are the same.

Then the \textbf{stretching class} determined by $\ol{\beta}$ consists of all functions in $\bigsqcup_n \R^{\st_n(G)}$ which assume the non-asterisked values at the prescribed places, and which assume the asterisked values along the stretched path in the prescribed order, each repeated any nonzero number of times.

\end{definition}

We note that Propositions \ref{stretchingroots} and \ref{squishing} imply that if one element of a stretching class is a root, they all are, so we can also think of stretching classes as subsets of $\bigsqcup_n \st_n(\Phi)$.

As we denote individual roots by Greek letters, we denote stretching classes by barred Greek letters, such as $\ol{\beta}$. We denote the set of functions in $\ol{\beta}$ defined on a specific stretch $\st_n(G)$ by $\ol{\beta}[n]$. We emphasize that, despite our name and notation, stretching classes are not equivalence classes, since they can intersect nontrivially.

In what follows we will consider how reflecting a root affects the stretching classes it belongs to. To do this, we need to distinguish whether this reflection is happening on or off the path of repeatable values marked by asterisks, or at the path's ends. 

We will freely talk about the vertices and coefficients of a stretching class, by which we mean the vertices and coefficients in the defining notation. In this language, we define the \textbf{left endpoint} of a stretching class to be the leftmost vertex with an asterisk, and define the \textbf{right endpoint} to be the rightmost such vertex. The \textbf{internal vertices} will be the other vertices with asterisks. The \textbf{class left neighbors} will be the neighbors to the left of the left endpoint, and we define the \textbf{class right neighbors} similarly. (Note that these may be different from $L_x$ and $R_x$, since in general not every vertex on the stretched path will have an asterisk.) We illustrate these terms in Figure \ref{stretchingclassterms}. 

\begin{figure}
\begin{center}
\begin{tikzpicture}
\node(1) at (-1, 0.5) {$3$};
\node(2) at (-1, 0) {$3$};
\node(3) at (-1, -0.5) {$3$};
\node[draw=black](40) at (0, 0) {$9$};
\node[draw=black](41) at (1, 0) {$9^*$};
\node[draw=black](42) at (2, 0) {$7^*$};
\node[draw=black](43) at (3, 0) {$2^*$};
\node(5) at (4, 0.5) {$1$};
\node(6) at (4, -0.5) {$1$};

\node(ln) at (0, 2) {class left neighbor};
\node(le) at (1, -2) {left endpoint};
\node(iv) at (2, 1) {internal vertex};
\node(re) at (3, -1) {right endpoint};
\node(rn) at (4, 2) {class right neighbors};

\draw (3.8, 0.7) rectangle (4.2, -0.7);

\draw (1) -- (40);
\draw (2) -- (40);
\draw (3) -- (40);
\draw (40) -- (41);
\draw (41) -- (42);
\draw (42) -- (43);
\draw (43) -- (5);
\draw (43) -- (6);
\draw[->] (ln) -> (40);
\draw[->] (le) -> (41);
\draw[->] (iv) -> (42);
\draw[->] (re) -> (43);
\draw[->] (rn) -> (5);
\end{tikzpicture}
\end{center}

\caption{An example of the terminology we use with stretching classes.}
\label{stretchingclassterms}
\end{figure}
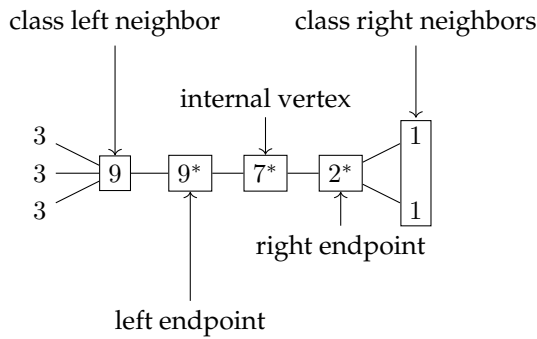

We may also talk about reflecting at a vertex of a stretching class, which amounts to applying the reflection to the defining notation as if it were an ordinary function on $\st_n(G)$, ignoring the asterisks. 

Finally, let $x_\ell$ be the left endpoint and let $y_1,\ldots, y_k$ be the class left neighbors. Then the \textbf{weighted left sum} of $\ol{\beta}$ is $\sum_{i} -A_{x_\ell y_i}\ol{\beta}(y_i)$. Note that this is the quantity $S_L$ from section \ref{stretchingandexpressions} if $x_\ell = x_0$, and $\ol{\beta}(x_{\ell - 1})$ otherwise. We likewise define the \textbf{weighted right sum}.

Now fix a positive root $\alpha$. We iteratively construct a directed graph $P$ whose vertices represent stretching classes and which will, for sufficiently large $n$, describe $\downset\st_n(\alpha)$. First, define a stretching class $\alpha^*$ by placing an asterisk on the value of $\alpha$ at $x$, so that $\alpha^*$ consists of the stretches of $\alpha$. We make $\alpha^*$ a vertex of $P$. Then in each step of constructing $P$, for each of its vertices $\ol{\beta}$, we add arrows $\ol{\beta}\to \ol{\gamma}$, where $\ol{\gamma}$ can be obtained from $\ol{\beta}$ by the following operations:

\begin{itemize}
\item[(1)] Reflect at a vertex without an asterisk, such that its coefficient decreases.
\item[(2)] Reflect at an internal vertex, such that its coefficient decreases.
\item[(3)] If the weighted left (right) sum is less than the coefficient of the left (right) endpoint, insert that sum with an asterisk as the new left (right) endpoint.
\item[(4)] If there is more than one vertex with an asterisk, reflect at the left/right endpoint such that the coefficient there decreases. If the new coefficient on the left/right endpoint is greater than or equal to the asterisked coefficient next to it, remove the asterisk from the left/right endpoint.
\end{itemize}

Figure \ref{stablegraphops} shows examples of these operations.

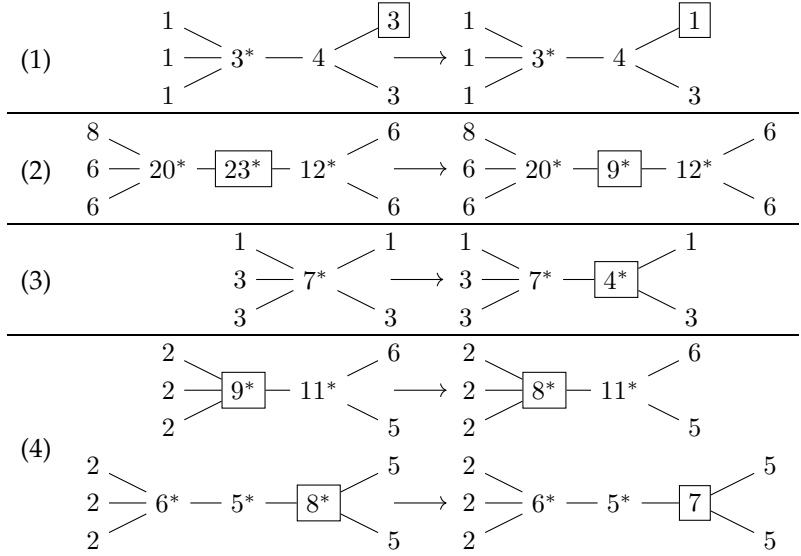
\begin{figure}
\begin{tabular}{lc}
(1) & \begin{tikzpicture}[baseline={([yshift=-.8ex]current bounding box.center)}]
\node(10) at (-1, 0.5) {$1$};
\node(20) at (-1, 0) {$1$};
\node(30) at (-1, -0.5) {$1$};
\node(400) at (0, 0) {$3^*$};
\node(410) at (1, 0) {$4$};
\node[draw=black](50) at (2, 0.5) {$3$};
\node(60) at (2, -0.5) {$3$};

\draw (10) -- (400);
\draw (20) -- (400);
\draw (30) -- (400);
\draw (400) -- (410);
\draw (410) -- (50);
\draw (410) -- (60);

\draw[->] (2, 0) -- (2.7, 0);

\node(11) at (3, 0.5) {$1$};
\node(21) at (3, 0) {$1$};
\node(31) at (3, -0.5) {$1$};
\node(401) at (4, 0) {$3^*$};
\node(411) at (5, 0) {$4$};
\node[draw=black](51) at (6, 0.5) {$1$};
\node(61) at (6, -0.5) {$3$};

\draw (11) -- (401);
\draw (21) -- (401);
\draw (31) -- (401);
\draw (401) -- (411);
\draw (411) -- (51);
\draw (411) -- (61);
\end{tikzpicture}
\\\hline
(2) & \begin{tikzpicture}[baseline={([yshift=-.8ex]current bounding box.center)}]
\node(10) at (-1, 0.5) {$8$};
\node(20) at (-1, 0) {$6$};
\node(30) at (-1, -0.5) {$6$};
\node(400) at (0, 0) {$20^*$};
\node[draw=black](410) at (1, 0) {$23^*$};
\node(420) at (2, 0) {$12^*$};
\node(50) at (3, 0.5) {$6$};
\node(60) at (3, -0.5) {$6$};

\draw (10) -- (400);
\draw (20) -- (400);
\draw (30) -- (400);
\draw (400) -- (410);
\draw (410) -- (420);
\draw (420) -- (50);
\draw (420) -- (60);

\draw[->] (3, 0) -- (3.7, 0);

\node(11) at (4, 0.5) {$8$};
\node(21) at (4, 0) {$6$};
\node(31) at (4, -0.5) {$6$};
\node(401) at (5, 0) {$20^*$};
\node[draw=black](411) at (6, 0) {$9^*$};
\node(421) at (7, 0) {$12^*$};
\node(51) at (8, 0.5) {$6$};
\node(61) at (8, -0.5) {$6$};

\draw (11) -- (401);
\draw (21) -- (401);
\draw (31) -- (401);
\draw (401) -- (411);
\draw (411) -- (421);
\draw (421) -- (51);
\draw (421) -- (61);
\end{tikzpicture}
\\\hline
(3) & \begin{tikzpicture}[baseline={([yshift=-.8ex]current bounding box.center)}]
\node(dummy) at (-2, 0) {};

\node(10) at (-1, 0.5) {$1$};
\node(20) at (-1, 0) {$3$};
\node(30) at (-1, -0.5) {$3$};
\node(400) at (0, 0) {$7^*$};
\node(50) at (1, 0.5) {$1$};
\node(60) at (1, -0.5) {$3$};

\draw (10) -- (400);
\draw (20) -- (400);
\draw (30) -- (400);
\draw (400) -- (50);
\draw (400) -- (60);

\draw[->] (1, 0) -- (1.7, 0);

\node(11) at (2, 0.5) {$1$};
\node(21) at (2, 0) {$3$};
\node(31) at (2, -0.5) {$3$};
\node(401) at (3, 0) {$7^*$};
\node[draw=black](411) at (4, 0) {$4^*$};
\node(51) at (5, 0.5) {$1$};
\node(61) at (5, -0.5) {$3$};

\draw (11) -- (401);
\draw (21) -- (401);
\draw (31) -- (401);
\draw (401) -- (411);
\draw (411) -- (51);
\draw (411) -- (61);
\end{tikzpicture}

\\\hline

(4) & \begin{tikzpicture}[baseline={([yshift=-.8ex]current bounding box.center)}]
\node(10) at (0, 0.5) {$2$};
\node(20) at (0, 0) {$2$};
\node(30) at (0, -0.5) {$2$};
\node[draw=black](400) at (1, 0) {$9^*$};
\node(410) at (2, 0) {$11^*$};
\node(50) at (3, 0.5) {$6$};
\node(60) at (3, -0.5) {$5$};

\draw (10) -- (400);
\draw (20) -- (400);
\draw (30) -- (400);
\draw (400) -- (410);
\draw (410) -- (50);
\draw (410) -- (60);

\draw[->] (3, 0) -- (3.7, 0);

\node(11) at (4, 0.5) {$2$};
\node(21) at (4, 0) {$2$};
\node(31) at (4, -0.5) {$2$};
\node[draw=black](401) at (5, 0) {$8^*$};
\node(411) at (6, 0) {$11^*$};
\node(51) at (7, 0.5) {$6$};
\node(61) at (7, -0.5) {$5$};

\draw (11) -- (401);
\draw (21) -- (401);
\draw (31) -- (401);
\draw (401) -- (411);
\draw (411) -- (51);
\draw (411) -- (61);

\node(10a) at (-1, -1) {$2$};
\node(20a) at (-1, -1.5) {$2$};
\node(30a) at (-1, -2) {$2$};
\node(400a) at (0, -1.5) {$6^*$};
\node(410a) at (1, -1.5) {$5^*$};
\node[draw=black](420a) at (2, -1.5) {$8^*$};
\node(50a) at (3, -1) {$5$};
\node(60a) at (3, -2) {$5$};

\draw (10a) -- (400a);
\draw (20a) -- (400a);
\draw (30a) -- (400a);
\draw (400a) -- (410a);
\draw (410a) -- (420a);
\draw (420a) -- (50a);
\draw (420a) -- (60a);

\draw[->] (3, -1.5) -- (3.7, -1.5);

\node(11a) at (4, -1) {$2$};
\node(21a) at (4, -1.5) {$2$};
\node(31a) at (4, -2) {$2$};
\node(401a) at (5, -1.5) {$6^*$};
\node(411a) at (6, -1.5) {$5^*$};
\node[draw=black](421a) at (7, -1.5) {$7$};
\node(51a) at (8, -1) {$5$};
\node(61a) at (8, -2) {$5$};

\draw (11a) -- (401a);
\draw (21a) -- (401a);
\draw (31a) -- (401a);
\draw (401a) -- (411a);
\draw (411a) -- (421a);
\draw (421a) -- (51a);
\draw (421a) -- (61a);
\end{tikzpicture}
\end{tabular}
\caption{The four operations we can perform on stretching classes corresponding to reflections on their roots.}
\label{stablegraphops}
\end{figure}

In any of these cases, we don't add the arrow if the operation results in a negative coefficient. This allows for the construction of $P$ to eventually stop, and we now show this happens.
\begin{lemma}
$P$ is finite and acyclic.
\end{lemma}
\begin{proof}
We track the following tuple of numbers associated to a stretching class in lexicographic order, from most to least significant:
\begin{itemize}
\item The sum of the weighted left and right sums.
\item The sum of the coefficients at the left and right endpoints.
\item The sum of all the coefficients.
\end{itemize}
We check that the operations used to define $P$ can only decrease this tuple.
\begin{itemize}
\item[(1,2)] Operations 1 and 2 don't lengthen the diagram, and they decrease one of the coefficients. Then either the first quantity decreases, or the first two stay the same while the third decreases.
\item[(3)] This operation leaves the class left/right neighbors untouched, but decreases the coefficient at either the left or right endpoint, so it keeps the first quantity the same while decreasing the second.
\item[(4)] If the reflected coefficient on the left endpoint is greater than or equal to the asterisked coefficient next to it, then it must also be less than the weighted left sum, or else the reflection would not have decreased it. By removing the left endpoint's asterisk, we make it the sole class left neighbor, and so we have decreased the weighted left sum.

On the other hand, if we don't remove the left endpoint's asterisk, then the sum of the weighted left and right sums stays the same while the sum of the left and right endpoints' coefficients decreases. 
\end{itemize}

Thus $P$ has no oriented cycles and (since we require every vertex to have all nonnegative coefficients) no infinite paths. Since each vertex has only finitely many arrows emanating from it, we also know $P$ is finite.
\end{proof}

Now we show the main result of this section.

\begin{theorem}
\label{stableposet}
Let the graph $P$ be constructed from a root $\alpha$ as above. Let $n_0$ be the smallest value such that every stretching class in $P$ with a single asterisk has an element defined on $\st_{n_0}(G)$. Then for $n\geq n_0 + 1$,
\[
\downset\st_n(\alpha) = \bigcup_{\ol{\beta}\in P} \ol{\beta}[n]
\]
\end{theorem}

We present each direction of containment as a separate lemma.

\begin{lemma}
Let $\alpha, P, n_0$ be as above, and $n \geq n_0 + 1$. Then
\[
\downset\st_n(\alpha) \subset \bigcup_{\ol{\beta}\in P} \ol{\beta}[n]
\]
\end{lemma}
\begin{proof}

Certainly $\st_n(\alpha)$ is in the latter set for any $n$. We then show that, for any cover relation $\delta \lessdot \gamma$ in the root poset, if $\gamma\in \bigcup_{\ol{\beta}\in P} \ol{\beta}[n]$, then so is $\delta$. Let $\ol{\gamma}$ be a stretching class in $P$ which contains $\gamma$. We claim that $\delta$ belongs to a stretching class obtained by applying to $\ol{\gamma}$ one of the operations used to define $P$, or to $\ol{\gamma}$ itself. 

We proceed by cases. Say that a coefficient of $\gamma$ is \textbf{repeatable} if it is represented by an asterisked vertex in $\ol{\gamma}$, and say that a repeatable coefficient is \textbf{alone} if it is the only coefficient of $\gamma$ represented by that vertex (so that both of its neighbors are different).
\begin{itemize}
\item[(0)] If $\delta$ is obtained from $\gamma$ by reflecting at a repeatable coefficient which is not alone, and it is not the furthest left or furthest right repeatable coefficient, then the reflection there changes the quantities of repeated coefficients but not which ones appear:

\begin{center}
\begin{tikzpicture}
\node (00) at (-2, 0) {$\cdots$};
\node (10) at (-1, 0) {$a$};
\node (20) at (0, 0) {$b$};
\node (30) at (1, 0) {$b$};
\node (40) at (2, 0) {$\cdots$};
\draw 	 (00) -- (10) -- (20) -- (30) -- (40);

\draw[->] (0, -0.5) -- (0, -1);

\node (01) at (-2, -1.5) {$\cdots$};
\node (11) at (-1, -1.5) {$a$};
\node (21) at (0, -1.5) {$a$};
\node (31) at (1, -1.5) {$b$};
\node (41) at (2, -1.5) {$\cdots$};
\draw 	 (01) -- (11) -- (21) -- (31) -- (41);
\end{tikzpicture}
\end{center}
Thus $\delta$ is also in $\ol{\gamma}$.

\item[(1)] If $\delta$ is obtained from $\gamma$ by reflecting at a non-repeatable coefficient, then $\delta$ belongs to a stretching class obtained by applying operation 1 above. 

\item[(2)] If we reflect at an alone coefficient other than the furthest left or furthest right repeatable coefficient, then $\delta$ lies in the stretching class obtained by applying operation 2 above.

\item[(3)] If we reflect at the furthest left or furthest right repeatable coefficient and it is not alone, then $\delta$ lies in the stretching class obtained by applying operation 3.

\item[(4)] If we reflect at the furthest left or furthest right repeatable coefficent and it is alone, then because $n\geq n_0 + 1$, $\ol{\gamma}$ must have more than one coefficient with an asterisk. Then $\delta$ lies in the stretching class obtained by applying operation 4.
\end{itemize}

\end{proof}

\begin{lemma} Let $\alpha, P$ be as above. Then
\[
\bigcup_{\ol{\beta}\in P} \ol{\beta}[n] \subset \downset\st_n(\alpha)
\]
\end{lemma}
In particular, this second containment is true for all $n$.
\begin{proof}
We know that $\st_n(\alpha)$ is the sole member of $\alpha^*[n]$, and it is in $\downset\st_n(\alpha)$. Then we will show that, for each arrow $\ol{\gamma}\to \ol{\delta}$ of $P$ and $\delta\in \ol{\delta}[n]$, there is some $\gamma\in \ol{\gamma}$ such that $\delta\leq \gamma$. We check this for each of the four operations:
\begin{itemize}
\item[(1)] If $\ol{\delta}$ is obtained from $\ol{\gamma}$ by a reflection at a non-asterisked vertex, then any root in $\ol{\delta}$ is obtained from one in $\ol{\gamma}$ with the same amount of each asterisked coefficient, just by performing that reflection.
\item[(2)] Suppose $\ol{\delta}$ is obtained from $\ol{\gamma}$ by a reflection at an internal vertex where $\ol{\delta}$ has coefficient $b$, and let $\delta \in \ol{\delta}[n]$. If $b$ is alone in $\delta$, then reflecting there will bring us up to a root in $\ol{\gamma}$. Otherwise, we know that some neighbor of $b^*$ in $\ol{\delta}$ must have a coefficient $b' > b$, or else $b$ could not be smaller than the initial value at its vertex. By repeatedly reflecting at the instances of $b$ which neighbor instances of $b'$, we decrease the number of repetitions of $b$ while moving up in the root poset:
\begin{center}
\begin{tikzpicture}
\node (00) at (-3, 0) {$\cdots$};
\node (10) at (-2, 0) {$b$};
\node (20) at (-1, 0) {$b'$};
\node (30) at (0, 0) {$b'$};
\node (40) at (1, 0) {$b'$};
\node (50) at (2, 0) {$\cdots$};
\draw 	 (00) -- (10) -- (20) -- (30) -- (40) -- (50);

\draw[<-] (-0.5, -0.5) -- (-0.5, -1);

\node (01) at (-3, -1.5) {$\cdots$};
\node (11) at (-2, -1.5) {$b$};
\node (21) at (-1, -1.5) {$b$};
\node (31) at (0, -1.5) {$b'$};
\node (41) at (1, -1.5) {$b'$};
\node (51) at (2, -1.5) {$\cdots$};
\draw 	 (01) -- (11) -- (21) -- (31) -- (41) -- (51);

\draw[<-] (-0.5, -2) -- (-0.5, -2.5);

\node (02) at (-3, -3) {$\cdots$};
\node (12) at (-2, -3) {$b$};
\node (22) at (-1, -3) {$b$};
\node (32) at (0, -3) {$b$};
\node (42) at (1, -3) {$b'$};
\node (52) at (2, -3) {$\cdots$};
\draw 	 (02) -- (12) -- (22) -- (32) -- (42) -- (52);
\end{tikzpicture}
\end{center}

Thus we reduce to the case that $b$ is alone.

\item[(3)] Suppose that $\ol{\delta}$ is obtained from $\ol{\gamma}$ by inserting the weighted (without loss of generality) left sum, which we call $b$, as the new left endpoint, and let $\delta\in \ol{\delta}[n]$. As in case (2), if $b$ is alone on the left end of the stretched path, reflecting there moves back up to an element of $\ol{\gamma}$, while if $b$ is not alone, the next coefficient to the right on the path will be larger, and we can move up to a case where $b$ is alone.

\item[(4)] Suppose that $\ol{\delta}$ is obtained from $\ol{\gamma}$ by reflecting at the (without loss of generality) left endpoint, resulting in the value $b$. Let $a$ be the weighted left sum, and let $c$ be the asterisked coefficient immediately to the right of $b$.

If $c\leq b$, then $b$ has no asterisk, so any $\delta\in \ol{\delta}[n]$ has only one instance of $b$ preceding $c$. Reflecting there returns us to a root in $\ol{\gamma}$. If $c > b$, $b$ may appear multiple times in $\delta$; however, since $c > b$, we can apply the same reasoning as in cases (2) and (3) to move up to an element of $\ol{\delta}[n]$ with only one instance of $b$, whereupon we fall back to the previous reasoning.
\end{itemize}
\end{proof}

Thus we can capture $\downset\st_n(\alpha)$ for sufficiently large $n$. We now derive the concrete consequence that $|\downset\st_n(\alpha)|$ is a polynomial in $n$.

First, we must deal with redundancy between our stretching classes, since they may nontrivally overlap. Fortunately, those overlaps are also stretching classes.

\begin{lemma}
The intersection of two stretching classes is either empty, a single root, or a stretching class.
\end{lemma}
\begin{proof}
If the classes assume different values off the stretched path, then their intersection is empty. Thus it suffices to assume the classes agree off the stretched path and consider the requirements they impose on the stretched path. It will clarify matters to introduce a slightly different notation.

For a symbol $a$ and positive integer $m$, let $a^{\geq m}$ denote the set of all words consisting of at least $m$ copies of $a$ and let $a^m$ denote the singleton set containing the word consisting of exactly $m$ copies of $a$. Then for multiple symbols $a_1,\ldots, a_k$ and integers $m_1,\ldots, m_k$, we denote by
\[
a_1^{(\geq) m_1}a_2^{(\geq) m_2}\cdots a_k^{(\geq) m_k}
\]
the set of all words obtained by concatenating words from these sets. 

In particular, by collapsing together repeated values, we see that the sequences of values which a stretching class allows to appear on the stretched path are described by an expression of the form
\[
a_1^{m_1}\cdots a_r^{m_r} a_{r+1}^{\geq m_{r+1}} a_{r+2}^{\geq 1} \cdots a_{s-1}^{\geq 1} a_s^{\geq m_s} a_{s+1}^{m_{s+1}}\cdots a_k^{m_k}
\]
in which no consecutive $a_i$'s are the same.

For the sets defined by two such expressions to intersect nontrivially, the values $a_1,\ldots, a_k$ must be the same. In this case, we can find their intersection by computing it for each $a_i$-value individually. We have
\begin{align*}
a^m \cap a^{m'} &= \begin{cases}
a^m & m = m' \\
\emptyset & m \neq m'
\end{cases} \\
a^{\geq m} \cap a^{m'} &= \begin{cases}
a^{m'} & m\leq m' \\
\emptyset & m > m'
\end{cases} \\
a^{\geq m} \cap a^{\geq m'} &= a^{\geq \max\{m, m'\}}
\end{align*}
Using these rules, one can check that intersecting two sets of the above form produces the empty set, a singleton (if no $a^{\geq m}$ terms remain), or another set of that form. 
\end{proof}

The final step is to compute the size of a single stretching class, which is a straightforward counting problem.
\begin{lemma}
Suppose $\ol{\beta}$ is a stretching class on $\st_m(G)$ with $\ell$ asterisked vertices. Then
\[
|\ol{\beta}[n]| = \binom{n - m + \ell - 1}{\ell - 1}
\]
\end{lemma}

Combining Theorem \ref{stableposet} with the last two lemmas and the inclusion-exclusion principle allows us to conclude:

\begin{theorem}
\label{polynomialgrowth}
Let $\alpha, P, n_0$ be as above. Let $\ell$ be the largest value such that there is a stretching class in $P$ with $\ell$ asterisked vertices. Then there is a polynomial $p(n)$ of degree $\ell - 1$ such that $|\downset\st_n(\alpha)| = p(n)$ for $n \geq n_0 + 1$. 
\end{theorem}

\section{Characteristic polynomials of shard arrangements}
\label{shards}

Recall that $V$ is the vector space containing the roots. Let $V^*$ be the dual space, and let $\<f, \alpha\> := f(\alpha)$ be the natural pairing $V^*\times V\to \R$. Then we consider the contragradient action of the Coxeter group $W$ on $V^*$, defined by $\<w(f), \alpha\> = \<f, w^{-1}(\alpha)\>$. For each root $\alpha$, $s_\alpha$ acts by a reflection over the hyperplane $\alpha^\perp := \{f\in V^*\mid \<f, \alpha\> = 0\}$.

It is in the context of these hyperplanes that we consider \textbf{rank 2 subarrangements}, the relation of \textbf{cutting}, the \textbf{fractures} obtained by intersecting each hyperplane with the ones cutting it, and the \textbf{shards} separated by these fractures, as described in Section \ref{shardsintro}. Our distinguished base region $D$ will be the subset of $V^*$ which pairs positively with every simple root, and thus with every positive root.

We will use an alternative description of fractures using reduced expressions, due to David Speyer and Hugh Thomas. First, we define a \textbf{rank 2 subsystem} of $\Phi$ to be the set of roots lying in a fixed 2-dimensional subspace which they span. These roots are dual to the hyperplanes of a rank 2 subarrangement, and we define the \textbf{fundamental roots} of the subsystem to be the positive roots corresponding to the basic hyperplanes.

\begin{lemma}
Let $R$ be a rank 2 subsystem. Let $\alpha, \beta$ be its fundamental roots. Then every positive root in $R$ is a nonnegative linear combination of $\alpha$ and $\beta$, and this property characterizes the fundamental roots.
\end{lemma}
\begin{proof}
Let $R^\perp$ be the associated rank 2 subarrangement. Let $D_R$ be the region of $R^\perp$ containing $D$. It consists of points which pair positively with every positive root in $R$.

In particular, because $\alpha^\perp$ and $\beta^\perp$ border $D_R$, if a point in $V^*$ pairs positively with $\alpha$ and $\beta$, it also pairs positively with every other positive root in $R$. If some other positive root $\gamma$ is a combination of $\alpha$ and $\beta$ with a negative coefficient, we can find a point which pairs positively with $\alpha$ and $\beta$ but not with $\gamma$, a contradiction.

Conversely, suppose every positive root in $R$ is a nonnegative linear combination of $\alpha$ and $\beta$. Then $D_R$ consists of points which pair positively with $\alpha$ and $\beta$, so it is bordered by $\alpha^\perp$ and $\beta^\perp$.
\end{proof}

\begin{lemma}
\label{fundamentaltofundamental}
Let $R$ be a rank 2 subsystem not containing $\alpha_i$. Suppose $\alpha$ and $\beta$ are the fundamental roots of $R$. Then $s_i\alpha$ and $s_i\beta$ are the fundamental roots of $s_iR$. 
\end{lemma}

\begin{proof}
Applying $s_i$ sends the set of positive roots other than $\alpha_i$ to itself. In particular, if every positive root in $R$ is a nonnegative linear combination of $\alpha$ and $\beta$, then every positive root in $s_iR$ is a nonnegative linear combination of $s_i\alpha$ and $s_i\beta$.
\end{proof}

\begin{proposition}
\label{fractureinductionstep} 
Let $\beta, \beta'$ be positive roots with $\beta = s_i\beta'$, such that $\beta - \beta'$ is a positive multiple of $\alpha_i$. Then the fractures of $\beta^\perp$ consist of $\alpha_i^\perp\cap\beta^\perp$ and all subspaces of the form $s_i(F')$ as $F'$ ranges over the fractures of $(\beta')^\perp$.
\end{proposition}

\begin{proof}
First, we show that $\alpha_i^\perp$ cuts $\beta^\perp$. Because $\beta$ is a linear combination of $\beta'$ and $\alpha_i$, they all lie in a rank 2 subsystem. $\alpha_i^\perp$ is a basic hyperplane, because it directly borders $D$. On the other hand, $\beta^\perp$ is not: if it were, $\beta'$ would be a positive linear combination of $\alpha_i$ and $\beta$, which is not true. Thus $\alpha_i^\perp$ intersects $\beta^\perp$ in a fracture.

Now let $R$ be any rank 2 subsystem containing $\beta$ other than the one just considered. By Lemma \ref{fundamentaltofundamental}, $s_i$ sends the fundamental roots of $R$ to those of $s_iR$. In particular, $R$ induces a fracture $\alpha^\perp\cap \beta^\perp$ if and only if $s_iR$ induces a fracture  $(s_i\alpha)^\perp\cap (\beta')^\perp$.
\end{proof}

\begin{proposition}
\label{fracturelist}
Suppose the root $\alpha$ has a reduced expression $s_{y_\ell}s_{y_{\ell-1}}\cdots s_{y_1}(\alpha_{y_0})$. Then the fractures of $\alpha^\perp$ are its intersections with
\begin{align*}
&\alpha_{y_\ell}^\perp  \\ 
&s_{y_\ell}(\alpha_{y_{\ell-1}})^\perp \\
&s_{y_\ell}s_{y_{\ell-1}}(\alpha_{y_{\ell-2}})^\perp \\
&\vdots \\
&s_{y_\ell}s_{y_{\ell-1}}\cdots s_{y_2}(\alpha_{y_1})^\perp
\end{align*}
\end{proposition}
\begin{proof}
By induction on depth. A simple root is fundamental in every rank 2 subsystem, so its hyperplane has no fractures. Now let $\alpha' = s_{y_{\ell-1}}\cdots s_{y_1}(\alpha_{y_0})$, and suppose the fractures of $(\alpha')^\perp$ are its intersections with
\begin{align*}
&\alpha_{y_{\ell-1}}^\perp  \\ 
&s_{y_{\ell-1}}(\alpha_{y_{\ell-2}})^\perp \\
&s_{y_{\ell-1}}s_{y_{\ell-2}}(\alpha_{y_{\ell-3}})^\perp \\
&\vdots \\
&s_{y_{\ell-1}}s_{y_{\ell-2}}\cdots s_{y_2}(\alpha_{y_1})^\perp
\end{align*}
By Proposition \ref{fractureinductionstep}, the list of fractures of $\alpha^\perp$ is obtained by applying $s_{y_\ell}$ to these and appending $\alpha_{y_\ell}^\perp\cap \alpha^\perp$, as required.
\end{proof}

In order to systematically describe the fractures of a root as we stretch it, it will be useful to systematically write reduced expressions for the stretches. Recall from section \ref{stretchingandexpressions} that the obstruction to this is case (3) of the trichotomy outlined there. In this case, just replacing each reflection at the elastic vertex will not produce a reduced expression. 

However, in the long run of stretching, we can avoid this obstruction. We may want to stretch some amount before choosing a reduced expression, so we first specify a way of adjusting our base diagram from $G$ to one of its stretches.

\begin{definition}
Let $G$ be a Coxeter diagram with elastic data $(x, L_x, R_x)$ and let $x_i$ be a vertex on the stretched path of $\st_{n_0}(G)$. Then the \textbf{elastic data induced by $x_i$} for $\st_{n_0}(G)$ is
\[
\begin{aligned}
&(x_0, L_x, \{x_1\}) && \text{ if } i = 0 \\
&(x_i, \{x_{i-1}\}, \{x_{i+1}\}) && \text{ if } 0 < i < n_0 \\
&(x_{n_0}, \{x_{n_0 - 1}\}, R_x) && \text{ if } i = n_0
\end{aligned}
\]
\end{definition}

\begin{lemma}
\label{systematicexpression}
Let $\alpha$ be a positive root. Then there exists some $n_0$, a vertex $x_i$ on the stretched path of $\st_n(G)$, and a reduced expression for $\st_{n_0}(\alpha)$ that avoids cases (2) and (3) of the above trichotomy with respect to the elastic data induced by $x_i$.
\end{lemma}
We will refer to such a reduced expression as a \textbf{type (1) expression} with respect to the induced elastic data.
\begin{proof}

We proceed by induction on the coefficient at $x$. In the base case that $\alpha(x) = 0$, we can avoid all cases of the trichotomy, because we never have to reflect at $x$. Now consider a root $\alpha$ with $\alpha(x) > 0$ and begin constructing a chain down from it in the root poset. Suppose that, having reached the root $\alpha'$, we reflect at $x$ for the first time. 

Suppose first that $\alpha'$ satisfies case (1) of the trichotomy. Using the induction hypothesis, we obtain some stretch $\st_{n_0}(s_x(\alpha'))$ and a type (1) expression for this root with respect to some $x_i$. We now use the following lemma:

\begin{lemma}
If there exists a chain from $\alpha$ down to $\beta$ in the root poset which avoids cases (2) and (3) of the trichotomy with respect to $x$, then there is likewise such a chain from $\st_n(\alpha)$ down to $\st_n(\beta)$ which avoids cases (2) and (3) with respect to any $x_i$ on the stretched path.
\end{lemma}

\begin{proof}
Given such a chain, we can obtain a chain from $\st_n(\alpha)$ down to $\st_n(\beta)$ by replacing every reflection at $x$ with the first half of the chain constructed in the proof of Lemma \ref{stretchedcovers}. Then each reflection at $x_i$ in this chain also falls under case (1).
\end{proof}

In this case, we have a chain from $\alpha$ down to $s_x(\alpha')$ which avoids cases (2) and (3). Applying the lemma, we get a similar chain from $\st_{n_0}(\alpha)$ down to $\st_{n_0}(s_x(\alpha'))$, which we can append to the type (1) expression for $\st_{n_0}(s_x(\alpha'))$ to get a type (1) expression for $\st_{n_0}(\alpha)$.

Now suppose instead that $\alpha'$ satisfies case (3). Assume without loss of generality that $S_L < \alpha(x)$, while $S_R > \alpha(x)$. Then $\st_1(\alpha')$ satisfies case (1) with respect to vertex $x_0$. By the induction hypothesis, there is some $n_0$ and a type (1) expression for $\st_{n_0}(s_{x_0}(\st_1(\alpha')))$ with respect to some $x_{0i}$. Then by the above lemma we can turn the chain from $\st_1(\alpha)$ down to $s_{x_0}(\st_1(\alpha'))$ into a chain from $\st_{n_0 + 1}(\alpha)$ down to $\st_{n_0}(s_{x_0}(\st_1(\alpha')))$, which we can append to the type (1) expression for $\st_{n_0}(s_{x_0}(\st_1(\alpha')))$ to get a type (1) expression for $\st_{n_0+1}(\alpha)$.

If $\alpha'$ satisfies case (2), we apply the same argument as in case (3).
\end{proof}

Once we have a type (1) expression for $\st_{n_0}(\alpha)$ with respect to the elastic data induced by $x_i$, we can get a reduced expression for any $\st_{n_0 + n}(\alpha) = \st_n(\st_{n_0}(\alpha))$ by replacing each instance of $s_{x_i}$ with an appropriate choice of $s_{x_{i0}}s_{x_{i1}}\cdots s_{x_{in}}$ or $s_{x_{in}}s_{x_{i(n-1)}}\cdots s_{x_{i0}}$, as in part (1) of Lemma \ref{stretchedcovers}. In particular, this construction implies:
\begin{proposition}
\label{whattis}
The number of reflections at the elastic vertex in a type (1) expression for $\alpha$ is the depth growth rate of $\alpha$.
\end{proposition}

Thus we can write down reduced expressions for stretched roots in a systematic way, and we will describe the fractures for stretched versions of a root in a systematic way.

To clarify which aspects of these arrangements remain stable, we take a cue from type $A$, and write the roots of $\st_n(\Phi)$ in a different basis, still indexed by the vertices of $G$. We define
\[
\beta_y := \begin{cases}
\alpha_y & y\notin \{x_0,\ldots, x_n\} \\
\sum_{j = 0}^i \alpha_{x_j} & y = x_i
\end{cases}
\]
or, inversely,
\[
\alpha_y = \begin{cases}
\beta_y & y\notin \{x_1,\ldots, x_n\} \\
\beta_{x_i} - \beta_{x_{i-1}} & y = x_i, i\geq 1
\end{cases}
\]

In doing this, we aim to recapture the interpretation of the type $A$ Coxeter groups as symmetric groups, with simple reflections corresponding to transpositions. To gauge the extent to which this works, we check how the simple reflections $s_{x_i}$ act in the $\beta$-basis. We have

\begin{equation}
\label{x0reflection}
s_{x_0}(\beta_y) = \begin{cases}
\beta_y + \beta_{x_0} & y\in L_x \\
-\beta_{x_0} & y = x_0 \\
\beta_{y} - \beta_{x_0} & y = x_i, i\geq 1 \\
\beta_y & \text{otherwise}.
\end{cases}
\end{equation}
For $1\leq j\leq n-1$, we have
\begin{equation}
\label{xireflection}
s_{x_j}(\beta_y) = \begin{cases}
\beta_{x_j} & y = x_{j-1} \\
\beta_{x_{j-1}} & y = x_j \\
\beta_y & \text{otherwise}
\end{cases}
\end{equation}
Finally, we have
\begin{equation}
\label{xnreflection}
s_{x_n}(\beta_y) = \begin{cases}
\beta_y + \beta_{x_n} - \beta_{x_{n-1}} & y\in R_x \\
\beta_{x_n} & y = x_{n-1} \\
\beta_{x_{n-1}} & y = x_n \\
\beta_y & \text{otherwise}
\end{cases}
\end{equation}
Thus, with some hiccups at the ends of the path, the $s_{x_i}$ mostly act by transpositions on our $\beta$-basis.

With this property in mind, we examine the fractures of $\st_n(\alpha)^\perp$ in the $\beta$-basis.

\begin{theorem}
\label{fractureform}
Let $\alpha$ be a positive root with a type (1) expression. Then there exist:
\begin{itemize}
\item a nonnegative integer $r$;
\item two lists of formal linear combinations $f_1,\ldots, f_s$ and $g_1,\ldots, g_t$ (where $t$ is the depth growth rate of $\alpha$) of the following terms:
\begin{align*}
&\beta_y \text{ for $y$ a vertex of $G$ other than $x$}, \\
&\beta_{x_0},\beta_{x_1},\ldots, \beta_{x_r}, \\
&\beta_{x_{n-r}}, \beta_{x_{n-r+1}},\ldots, \beta_{x_n};
\end{align*}
\end{itemize}
such that for $n\geq 2r$, the fractures associated to $\st_n(\alpha)$ are precisely the intersections of $\st_n(\alpha)^\perp$ with the arrangement
\[
\{(f_i)^\perp\mid 1\leq i\leq s\}\cup \{(g_i - \beta_{x_j})^\perp \mid 1\leq i\leq t, r+1\leq j\leq n-r-1\}.
\]
\end{theorem}

\begin{proof}

We proceed by induction on the length of our type (1) expression. For simple roots, the proposition is vacuously true. Bearing in mind our above discussion of how to obtain a reduced expression for $\st_n(\alpha)$, it remains to show that the proposed uniform description of the fractures is preserved when we apply $s_y$ for $y\neq x$, as well as when we apply $s_{x_0}s_{x_1}\cdots s_{x_n}$ or $s_{x_n}s_{x_{n-1}}\cdots s_{x_0}$.

In the first case, applying a reflection $s_y$ for $y$ off the stretched path does not change the coefficients of $\alpha_{x_1},\ldots, \alpha_{x_{n-1}}$, because the vertices $x_1,\ldots, x_{n-1}$ do not neighbor any vertices off the stretched path. Thus it also does not change the coefficients of $\beta_{x_1},\ldots, \beta_{x_{n-2}}$ in the $\beta$-basis. Thus, though we may have to adjust $r$, applying $s_y$ to a list of fractures with the claimed uniform description gives another list with such a description. To this we add $\alpha_y^\perp$, which is unsupported at all the $\beta_{x_i}$, and thus also fits into the uniform description.

In particular, this doesn't increment the number $t$ of collections of fractures of the form $(g_i - \beta_j)^\perp$, which is consistent with the number of reflections at $x$ in the original expression (and thus, by Proposition \ref{whattis}, the depth growth rate) staying the same.

Thus it remains to show that applying the sequence $s_{x_0}\cdots s_{x_n}$ or $s_{x_n}\cdots s_{x_0}$ also preserves the uniform description of the fractures, using equations \ref{x0reflection}--\ref{xnreflection} from above.

First, if we apply $s_{x_0}\cdots s_{x_n}$ to a fracture which is unsupported at $\beta_{x_{r+1}},\ldots, \beta_{x_{n-r-1}}$, then it is straightforward to verify that the result is unsupported at $\beta_{x_{r+2}},\ldots, \beta_{x_{n-r}}$. 

Similarly, if we apply $s_{x_0}\cdots s_{x_n}$ to a fracture of the form $(g_i - \beta_{ x_j})^\perp$, where $g_i$ is unsupported at $\beta_{x_{r+1}},\ldots, \beta_{x_{n-r-1}}$, we get one of the form $(g'_i - \beta_{x_{j+1}})^\perp$, where $g'_i$ is unsupported at $\beta_{x_{r+2}},\ldots, \beta_{x_{n-r}}$. Again, after adjusting $r$, this matches the form required of the second class of fractures in our uniform description (with a couple now falling into the first class of fractures because of the change in $r$).

Thus applying $s_{x_0}\cdots s_{x_n}$ to the existing fractures gives a collection of fractures of the claimed form. To this collection we add
\begin{align*}
\alpha_{x_0}^\perp &= \beta_{x_0}^\perp \\
s_{x_0}(\alpha_{x_{1}})^\perp &= \beta_{x_1}^\perp \\
\vdots & \\
s_{x_0}s_{x_{1}}\cdots s_{x_{n-1}}(\alpha_{x_n})^\perp &= \beta_{x_n}^\perp
\end{align*}
which certainly have the $(g_i - \beta_j)^\perp$ form.

We can show in the same way that applying $s_{x_n}s_{x_{n-1}}\cdots s_{x_0}$ to the existing fractures gives a collection of fractures of the claimed form, and to this collection we add
\begin{align*}
\alpha_{x_n}^\perp &= (\beta_{x_n} - \beta_{x_{n-1}})^\perp \\
s_{x_n}(\alpha_{x_{n-1}})^\perp &= (\beta_{x_n} - \beta_{x_{n-2}})^\perp \\
\vdots & \\
s_{x_n}s_{x_{n-1}} \cdots s_{x_{2}}(\alpha_{x_{1}})^\perp &= (\beta_{x_n} - \beta_{x_0})^\perp \\ 
s_{x_n}s_{x_{n-1}}\cdots s_{x_{1}}(\alpha_{x_0})^\perp &= \beta_{x_n}^\perp
\end{align*}
which, except for the last one (which we can account for by adjusting $r$), have the $(g_i - \beta_j)^\perp$ form.

In either case, note that this increments the number $t$ of collections of fractures of the form $(g_i - \beta_j)^\perp$, which is consistent with the number of reflections at $x$ in the original expression being incremented.
\end{proof}

Combining this result with Lemma \ref{systematicexpression}, we draw a conclusion for arbitrary roots:
\begin{corollary}
Let $\alpha$ be any positive root. Then for sufficiently large $n$, the fractures of $\st_n(\alpha)^\perp$ admit a uniform description as in Theorem \ref{fractureform}. 
\end{corollary}

Now we use this description of the arrangements of fractures to describe their characteristic polynomials (as defined in \cite{athanasiadis}).

\begin{theorem}
\label{charpoly}
Suppose the positive root $\alpha$ has depth growth rate $t$. Let $\chi_n(q)$ be the characteristic polynomial of the fracture arrangement of $\st_n(\alpha)$. Then there exist polynomials $p_0(q),\ldots, p_t(q)$ and an integer $e$ such that
\[
\chi_n(q) = \sum_{k=1}^t p_k(q)(q-k)^{n-e}
\]
\end{theorem}

\begin{proof}

We use the following result of Athanasiadis:
\begin{lemma}[{\cite[Theorem 2.2]{athanasiadis}}]
\label{athanasiadisthm}
Let $\AA$ be any subspace arrangement defined over the integers and let $\chi(q)$ be its characteristic polynomial. Then for $q$ a sufficiently large prime, $\chi(q)$ is the number of points in the complement of $\AA$ over $\F_q$.
\end{lemma}

Thus, to show $\chi_n(q)$ has the claimed form, it will suffice to show that the point counts over $\F_q$ eventually do.

Choosing a point in the complement of the hyperplanes in Theorem \ref{fractureform} amounts to: 
\begin{itemize}
\item choosing all the coordinates except $\beta_{x_{r+1}},\ldots, \beta_{x_{n-r-1}}$ such that the $f_i$ are nonzero;
\item plugging these coordinates into the $g_i$ and choosing $\beta_{x_{r+1}},\ldots, \beta_{x_{n-r-1}}$ independently, subject to the condition that they are different from $g_i$, which excludes at most $t$ values.
\end{itemize}

To count these points, consider the subarrangement formed by the $f_i$. We stratify its complement according to the number of distinct values assumed by the $g_j$. Each stratum is built up from the hyperplanes $(f_i)^\perp$ and $(g_{j_1} - g_{j_2})^\perp$ through complementation, union, and intersection, and we so we can repeatedly use Lemma \ref{athanasiadisthm} to conclude that the number of points in the $k$th stratum over $\F_q$ is given by a polynomial $p_k(q)$ for large primes $q$. Combining this with the choice of the remaining variables gives the claimed form for $\chi_n(q)$. 

\end{proof}


\begin{corollary}
\label{exponentialgrowth}
Let $d$ be the number of vertices of $G$. Then for sufficiently large $n$, the number of shards of $\st_n(\alpha)^\perp$ is
\[
(-1)^{d - e - 1}\sum_{k=1}^t p_k(-1)(k+1)^{n-e}
\]
In particular, it is $O((t+1)^n)$.
\end{corollary}
\begin{proof}
This follows from the fact that the number of regions of a hyperplane arrangement in $\R^{d + n - 1}$ with characteristic polynomial $\chi(q)$ is $(-1)^{d+n-1}\chi(-1)$ \cite[Theorem 1.1]{athanasiadis}.
\end{proof}


\begin{thebibliography}{8}
\bibitem{armstrong} Drew Armstrong, \textit{Generalized noncrossing partitions and combinatorics of Coxeter groups}. Memoirs of the American Mathematical Society, {\bf 202}, no.\ 949, American Mathematical Society, Providence, 2009.

\bibitem{athanasiadis} Christos Athanasiadis, ``Characteristic polynomials of subspace arrangements and finite fields.'' Adv. in Math. \textbf{122} (1996), no.\ 2, 193--233.

\bibitem{bjornerbrenti} Anders Bjorner and Francesco Brenti, \textit{Combinatorics of Coxeter Groups}. Graduate Texts in Mathematics, {\bf 231}, Springer, New York, 2005.

\bibitem{hepworth} Richard Hepworth, ``Homological stability for families of Coxeter groups''. Algebr.\ Geom.\ Topol.\ \textbf{16} (2016), no.\ 5, 2779--2811.

\bibitem{reading1} Nathan Reading, ``Lattice theory of the poset of regions''. In: George Gr\"atzer and Friedrich Wehrung (eds.), \textit{Lattice Theory: Special Topics and Applications}, Volume 2, Birkh\"auser, Cham, 2016.

\bibitem{reading2} Nathan Reading, ``Finite Coxeter groups and the weak order''. In: George Gr\"atzer and Friedrich Wehrung (eds.), \textit{Lattice Theory: Special Topics and Applications}, Volume 2, Birkh\"auser, Cham, 2016.

\bibitem{thomas} Hugh Thomas, ``Stability, shards, and preprojective algebras''. Preprint 2017. (arXiv:1706.00164)

\end{thebibliography}
\end{document}